\documentclass[12pt,reqno]{amsart}

\usepackage{aliascnt,amsmath,amssymb,amsthm,amsfonts,enumerate,mathrsfs,mathtools,bm,xcolor,comment}
\usepackage[abbrev]{amsrefs}
\usepackage{shuffle}
\usepackage[OT2,T1]{fontenc}
\usepackage[marginparwidth=0pt,margin=24truemm]{geometry}

\definecolor{mylinkcolor}{RGB}{16,156,81}
\definecolor{mycitecolor}{RGB}{20,80,140}

\usepackage{hyperref}
\usepackage[nameinlink]{cleveref}
\hypersetup{
  setpagesize=false,
  bookmarksnumbered=true,
  bookmarksopen=true,
  colorlinks=true,
  linkcolor=mylinkcolor,
  citecolor=mycitecolor,
  urlcolor=mycitecolor
}

\numberwithin{equation}{section}
\allowdisplaybreaks[2]
\everymath{\displaystyle}

\theoremstyle{plain}
\newtheorem{thm}{Theorem}[section]
\crefname{thm}{Theorem}{Theorems}

\newtheorem{thmA}{Theorem}

\crefname{thmA}{Theorem}{Theorems}

\newaliascnt{prop}{thm}
\newtheorem{prop}[prop]{Proposition}
\aliascntresetthe{prop}
\crefname{prop}{Proposition}{Propositions}

\newaliascnt{lem}{thm}
\newtheorem{lem}[lem]{Lemma}
\aliascntresetthe{lem}
\crefname{lem}{Lemma}{Lemmas}

\newaliascnt{cor}{thm}

\aliascntresetthe{cor}
\crefname{cor}{Corollary}{Corollaries}

\theoremstyle{definition}
\newaliascnt{dfn}{thm}
\newtheorem{dfn}[dfn]{Definition}
\aliascntresetthe{dfn}
\crefname{dfn}{Definition}{Definitions}

\newaliascnt{rem}{thm}
\newtheorem{rem}[rem]{Remark}
\aliascntresetthe{rem}
\crefname{rem}{Remark}{Remarks}

\newcommand{\QQ}{\mathbb Q}
\newcommand{\eps}{\varepsilon}
\newcommand{\bi}[2]{\genfrac{(}{)}{0pt}{}{#1}{#2}}
\newcommand{\BARI}{\operatorname{BARI}}
\newcommand{\BIMU}{\operatorname{BIMU}}
\newcommand{\ARI}{\operatorname{ARI}}
\newcommand{\swapop}{\operatorname{swap}}
\newcommand{\pushop}{\operatorname{push}}
\newcommand{\muop}{\operatorname{mu}}
\newcommand{\invmu}{\operatorname{invmu}}
\newcommand{\gaxit}{\operatorname{gaxit}}
\newcommand{\garit}{\operatorname{garit}}
\newcommand{\gaxi}{\operatorname{gaxi}}
\newcommand{\arit}{\operatorname{arit}}
\newcommand{\anit}{\operatorname{anit}}
\newcommand{\axit}{\operatorname{axit}}
\newcommand{\preari}{\operatorname{preari}}
\newcommand{\ari}{\operatorname{ari}}
\newcommand{\uri}{\operatorname{uri}}
\newcommand{\ganit}{\operatorname{ganit}}
\newcommand{\gari}{\operatorname{gari}}
\newcommand{\invgari}{\operatorname{invgari}}
\newcommand{\expari}{\operatorname{expari}}
\newcommand{\logari}{\operatorname{logari}}
\newcommand{\adari}{\operatorname{adari}}
\newcommand{\pic}{\operatorname{pic}}
\newcommand{\poc}{\operatorname{poc}}
\newcommand{\pac}{\operatorname{pac}}
\newcommand{\antiop}{\operatorname{anti}}
\newcommand{\negop}{\operatorname{neg}}
\newcommand{\asna}{\operatorname{asna}}
\newcommand{\one}{\mathbf 1}
\newcommand{\cP}{\mathcal P}
\newcommand{\cL}{\mathcal L}
\newcommand{\cD}{\mathcal D}
\newcommand{\leadop}{\operatorname{lead}}
\newcommand{\Lbial}{\ARI_{\underline{\mathrm{al}}/\underline{\mathrm{al}}}}
\newcommand{\BARIswapil}{\BARI_{\swapop,\underline{\mathrm{il}}}}
\newcommand{\BARIpol}{\BARI^{\mathrm{pol}}}

\title[Lie algebras for multiple Eisenstein series]{Lie algebras for multiple Eisenstein series\\ and multiple $q$-zeta values}

\author{Henrik Bachmann}
\address{Graduate School of Mathematics, Nagoya University, Nagoya, Japan.}
\email{henrik.bachmann@math.nagoya-u.ac.jp}

\author{Hanamichi Kawamura}
\address{Department of Mathematics, Graduate School of Science, Tokyo University of Science, 1-3 Kagurazaka, Shinjuku-ku, Tokyo, Japan.}
\email{1125512@ed.tus.ac.jp}

\date{\today}
\subjclass[2020]{Primary 11M32; Secondary 17B01}
\keywords{bimoulds, uri bracket, alternility, swap invariance, multiple Eisenstein series, multiple $q$-zeta values}

\begin{document}

\begin{abstract}
Racinet's double shuffle Lie algebra \(\mathfrak{dm}_0\) encodes the double shuffle relations of multiple zeta values.  We study two analogues of \(\mathfrak{dm}_0\) for multiple Eisenstein series and \(q\)-analogues of multiple zeta values, namely the space \(\BARIswapil\) of swap-invariant alternil bimoulds with the uri bracket of K\"uhn and Schneps, and Burmester's balanced double shuffle space \(\mathfrak{bm}_0\).  Both were conjectured to be Lie algebras.  We prove both conjectures and show that \(\mathfrak{bm}_0\) is isomorphic to the Lie algebra of finite-depth polynomial elements of \(\BARIswapil\).
\end{abstract}

\maketitle

\section{Introduction}

Racinet's double shuffle Lie algebra \(\mathfrak{dm}_0\) \cite{Rac} controls the double shuffle relations of multiple zeta values.  In this note we study the space \(\BARIswapil\) of swap-invariant alternil bimoulds with the uri bracket and Burmester's balanced double shuffle space \(\mathfrak{bm}_0\).  They are analogues of \(\mathfrak{dm}_0\) for the algebra of multiple Eisenstein series and for \(q\)-analogues of multiple zeta values.  Both were conjectured to be Lie algebras, and we prove both conjectures.

The uri bracket on bimoulds was proposed by K\"uhn and Schneps and studied systematically by Burmester in \cite{Bu1}.  Conjugating \'{E}calle's ari bracket \cite{Ec1} with two rational flexion operators (see \cref{def:uri}), one gets a Lie bracket on alternil bimoulds with poles, but it is not clear from the definition that these poles cancel for power-series or polynomial inputs.  Burmester and K\"uhn identify the ari bracket with a post-Lie structure on a free Lie algebra and conjecture an analogous description for the uri bracket \cite{BuK}.

We write \(\BARIswapil\) for the space of power-series bimoulds which are alternil, swap invariant and even in depth one (see \Cref{sec:bimoulds}), and \(\mathfrak B\) for its subspace of finite-depth polynomial bimoulds over \(\QQ\) (see \Cref{sec:polynomial}).  Our first main result is the following.

\begin{thmA}\label{thm:A}
\begin{enumerate}[(i)]
\item The space \(\BARIswapil\) is a Lie algebra with the uri bracket, i.e.\ we have
\[
 \uri(\BARIswapil,\BARIswapil)\subseteq\BARIswapil.
\]
\item The space \(\mathfrak B\) is a weight-graded Lie subalgebra of \(\BARIswapil\).
\end{enumerate}
\end{thmA}

The statement of \cref{thm:A} was conjectured by K\"uhn and Schneps (see \cite{Bu1}*{Conjecture~5.19}), who studied the polynomial space and bracket underlying it and obtained partial results toward the closure assertion.\footnote{Their work was not published.  See K\"uhn's talk slides \cite{Kue} for an account.}  Burmester constructed a uri Lie subalgebra of the polynomial elements of \(\BARIswapil\) in \cite{Bu1}*{Theorem~5.20}.  Further properties of \(\BARIswapil\) and the associated depth-graded Lie algebra are studied in \cite{BKS}.

The space \(\mathfrak B\) also has a description in terms of noncommutative polynomials.  For this, we use Burmester's balanced double shuffle space \(\mathfrak{bm}_0\), introduced in \cite{Bu1}.  It is the analogue of \(\mathfrak{dm}_0\) for formal multiple Eisenstein series.  These series can be seen as the weight-graded version of multiple \(q\)-zeta values (see \cites{Bu2,BIM}).  Burmester also constructed an embedding of vector spaces \(\theta:\mathfrak{dm}_0\hookrightarrow\mathfrak{bm}_0\) and an explicit bracket \(\{\ ,\ \}_{\mathrm b}\) on the ambient space.  She conjectured that \(\mathfrak{bm}_0\) is closed under this bracket \cite{Bu1}*{Conjectures~4.26(i) and~4.52}.  Her bimould realization gives a vector-space isomorphism \(\Phi:\mathfrak{bm}_0\to\mathfrak B\) \cite{Bu1}*{Corollary~5.51}.  Our second main result is the following.

\begin{thmA}\label{thm:B}
The space \(\mathfrak{bm}_0\) is a weight-graded Lie algebra under the bracket \(\{\ ,\ \}_{\mathrm b}\), and
\[
 \Phi:(\mathfrak{bm}_0,\{\ ,\ \}_{\mathrm b})
 \overset{\sim}{\longrightarrow}(\mathfrak B,\uri)
\]
is an isomorphism of weight-graded Lie algebras.
\end{thmA}

The space \(\mathfrak{bm}_0\) is attached to the algebra \(\mathcal G^{\!f}\) of formal multiple Eisenstein series introduced in \cites{Bu2,BIM}.  The elements of \(\mathcal G^{\!f}\) are swap invariant and their product is the bi-stuffle product.  Therefore the generating series of the formal multiple Eisenstein series is a symmetril and swap-invariant bimould over \(\mathcal G^{\!f}\) (see \cite{BIM}*{Section~2}).  Alternility and swap invariance are the corresponding tangent conditions at \(\one\).  Together with the parity condition in depth one, they define \(\BARIswapil\), and \(\mathfrak B\) is its finite-depth polynomial part.  The formal quasimodular subalgebra of \(\mathcal G^{\!f}\) is isomorphic to the algebra \(\widetilde{\mathcal M}\) of quasimodular forms \cite{BIM}*{Theorem~5.4}, and the defining relations of \(\mathcal G^{\!f}\) are the graded relations underlying \cite{BK}*{Conjecture~2.6}.  For multiple zeta values one expects an algebra isomorphism
\[
 \mathcal Z\overset{?}{\cong}\QQ[\zeta(2)]\otimes
 \mathcal U(\mathfrak{dm}_0)^\vee\,,
\]
where \(\mathcal U(\mathfrak{dm}_0)^\vee\) denotes the graded dual of the universal enveloping algebra of \(\mathfrak{dm}_0\), and for formal multiple zeta values this is a theorem of Racinet \cite{Rac} (see \cite{Bu1}*{Corollary~B.32}).  This suggests the analogue
\begin{equation}\label{eq:Gf-conjecture}
 \mathcal G^{\!f}
 \overset{?}{\cong}
 \widetilde{\mathcal M}\otimes\mathcal U(\mathfrak{bm}_0)^\vee,
\end{equation}
in which the quasimodular forms take the place of \(\zeta(2)\).  In terms of Hilbert series, the dimension conjecture of \cite{BK}*{Conjecture~1.3} can then be written as
\begin{align}\label{eq:Gf-dimension-conjecture}
 \sum_{k\geq0}\dim_\QQ\mathcal G^{\!f}_kX^k
 &\overset{?}{=}
 \widetilde{\mathsf M}(X)
 \frac{1}{1-\mathsf D(X)(X+\mathsf O(X))
              +\mathsf D(X)\mathsf W(X)}\notag\\
 &=\frac{1}{1-X-X^2-X^3+X^6+X^7+X^8+X^9},
\end{align}
where
\[
 \mathsf D(X)=\frac1{1-X^2},\qquad
 \mathsf O(X)=\frac{X^3}{1-X^2},\qquad
 \mathsf M(X)=\frac1{(1-X^4)(1-X^6)},
\]
\[
 \mathsf S(X)=X^{12}\mathsf M(X),\qquad
 \widetilde{\mathsf M}(X)=\mathsf D(X)\mathsf M(X),\qquad
 \mathsf W(X)=\mathsf M(X)+\mathsf S(X)-1.
\]
Here \(\mathsf W(X)\) is the Hilbert--Poincar\'e series of period polynomials.  For comparison, Zagier's dimension conjecture for multiple zeta values reads
\[
 \sum_{k\geq0}\dim_\QQ\mathcal Z_kX^k
 \overset{?}{=}\mathsf D(X)\,\frac{1}{1-\mathsf O(X)}
 =\frac{1}{1-X^2-X^3}\,,
\]
where \(\mathsf D(X)\) counts the powers of \(\zeta(2)\) and \(\mathsf O(X)\) counts the conjectural generators of \(\mathfrak{dm}_0\), one in each odd weight \(k\geq3\).  In \eqref{eq:Gf-dimension-conjecture} the factor \(\widetilde{\mathsf M}(X)\) replaces \(\mathsf D(X)\).  Instead of one generator in each odd weight, \(\mathsf D(X)(X+\mathsf O(X))\) counts the generators \(\xi_k^m=\delta^m(\xi_k)\), for odd \(k\geq1\) and \(m\geq0\).  Here \(\delta\) is the derivation of \((\mathfrak B,\uri)\) from \cite{Ba} which raises the weight by two (see \eqref{eq:delta}), so these are the \(\xi_k\) together with all their derivatives.  The factor \(\mathsf D(X)\mathsf W(X)\) counts the relations coming from period polynomials and all their derivatives \cite{Bu1}*{Conjecture~1.22(ii)--(iii) and Equation~(4.56.1)}.  In \Cref{sec:examples} we compute \(\mathfrak B\) in weight at most nine and list explicit elements.

For \cref{thm:A}, alternility is preserved after localization and \eqref{eq:diamond-depth-one} makes the depth-one condition automatic.  It therefore remains to prove regularity and swap invariance.  Regularity in every depth follows from the divided-difference recursion in \Cref{sec:regular,sec:conjugation}, which also appears in the joint work of the first author and Burmester \cite{BB2}.  To prove swap invariance, we use the dimorphic transport isomorphisms attached to the symmetric flexion unit \(1/u+1/v\), established by the second author in \cite{Kaw2}, to show that \(\mathscr J=T^{-1}\circ\swapop\circ T\) with \(T=\ganit_{\pic}\) is an automorphism of a dimorphic ari Lie algebra, whose fixed locus is closed under ari.  Notice that the isomorphism \(\Phi\) alone does not prove \cref{thm:B}, since the bimould realization forgets every word beginning in \(b_0\).  In \Cref{sec:balanced} we show that these coefficients are determined and compare the bracket \(\{\ ,\ \}_{\mathrm b}\) with uri using the derivation formula of \cite{Bu1}.

We will not talk about multiple Eisenstein series and \(q\)-analogues of multiple zeta values themselves in this paper.  For these we refer to \cite{BK} for \(q\)-analogues of multiple zeta values, to \cite{BB1} for combinatorial multiple Eisenstein series, to \cite{BKM} for relations among multiple Eisenstein series and their derivatives, and to \cites{Bu2,BIM} for the formal version.

\medskip
\noindent\textbf{Acknowledgements.}  The authors would like to thank Ulf K\"uhn for comments on an early version of this paper.  This project was partially supported by JSPS KAKENHI Grant 26K22254.

\section{Bimoulds and coefficient symmetries}\label{sec:bimoulds}

Let \(R\) be a commutative \(\QQ\)-algebra.  A bimould over \(R\) is a family
\[
 A=(A_r)_{r\geq0},\qquad
 A_r\in R[[X_1,Y_1,\ldots,X_r,Y_r]].
\]
We write \(\BARI\) for the space of bimoulds with \(A_0=0\).  The flexion operators which define the uri bracket produce denominators, and we therefore also need bimoulds in which poles are allowed.  The space of such bimoulds is denoted by \(\BARI^{\mathrm{loc}}\), with the poles specified in \eqref{eq:poles} below.  A bimould is \emph{normalized} if \(A_0=1\).  The normalized bimould with all positive-depth components equal to zero is denoted by \(\one\).

Write
\begin{equation}\label{eq:coefficient-expansion}
 A_r\bi{X_1,\ldots,X_r}{Y_1,\ldots,Y_r}
 =\sum_{\substack{k_1,\ldots,k_r\geq1\\m_1,\ldots,m_r\geq0}}
 a\bi{k_1,\ldots,k_r}{m_1,\ldots,m_r}
 \prod_{i=1}^r\frac{X_i^{k_i-1}Y_i^{m_i}}{m_i!}.
\end{equation}
Let \(y_{k,m}\) be a letter of weight \(k+m\).  The bi-stuffle product is defined recursively by \(1*w=w*1=w\) and
\begin{align*}
 y_{k,m}u*y_{l,n}v
 &={}
 y_{k,m}(u*y_{l,n}v)
 +y_{l,n}(y_{k,m}u*v)
 +y_{k+l,m+n}(u*v).
\end{align*}
We also write \(\shuffle\) for the ordinary shuffle product on the same alphabet.  The coefficient map sends \(y_{k_1,m_1}\cdots y_{k_r,m_r}\) to the coefficient displayed in \eqref{eq:coefficient-expansion}.  A bimould in \(\BARI\) is \emph{alternal} if this map vanishes on \(u\shuffle v\) for every pair of nonempty words \(u,v\).  It is \emph{alternil} if the same condition holds with \(*\) instead of the shuffle product.  A normalized bimould is \emph{symmetral} or \emph{symmetril} if its coefficient map is a character for the shuffle or bi-stuffle product, respectively.  Notice that the tangent conditions at \(\one\) are alternality and alternility.  For bimoulds with poles, let
\[
 A(u\shuffle v)=\sum_{\sigma\in\operatorname{Sh}(p,q)}
 A_{p+q}(w_{\sigma(1)},\ldots,w_{\sigma(p+q)})
\]
for words \(u=w_1\cdots w_p\) and \(v=w_{p+1}\cdots w_{p+q}\).  We define alternality by \(A(u\shuffle v)=0\), and symmetrality by \(A(u\shuffle v)=A(u)A(v)\).  Alternility and symmetrility extend to bimoulds with poles through the corresponding flexed bi-stuffle identities \cite{Bu1}*{Remarks~C.15(i) and~C.18}.

\'{E}calle's swap is the involution given in our coordinates by
\[
 \swapop(A)_r\bi{X_1,\ldots,X_r}{Y_1,\ldots,Y_r}
 =A_r\bi{Y_1+\cdots+Y_r,\ldots,Y_1}
 {X_r,X_{r-1}-X_r,\ldots,X_1-X_2}.
\]
We use \(\BARIswapil\) for the alternil bimoulds satisfying
\begin{equation}\label{eq:swap-parity}
 \swapop(A)=A,\qquad
 A_1\bi{-X_1}{-Y_1}=A_1\bi{X_1}{Y_1}.
\end{equation}
Following \'{E}calle's convention \cite{Sc}*{Section~2.5}, the underline indicates the depth-one condition in \eqref{eq:swap-parity}.  In the notation of \cite{Bu1}*{Theorem~5.2} the space \(\BARIswapil\) is \(\BARI_{\mathrm{il},\swapop}\).

\section{GAXI, GARI and the uri bracket}\label{sec:flexion}

We use the flexion conventions recalled in \cite{Sc} and \cite{Kaw1}.  With these conventions the tangent bracket of GARI is ari.  The terms \emph{symmetril} and \emph{alternil} refer to the standard flexion unit \(1/X\).

All flexion operators below produce denominators which are products of the linear forms \(X_i\) and \(X_i-X_j\).  We therefore work in the ring
\begin{equation}\label{eq:poles}
 \cL_r(R)
 :=R[[X_1,\ldots,X_r,Y_1,\ldots,Y_r]]
   \Bigl[\tfrac{1}{X_i},\tfrac{1}{X_i-X_j}\ \Big|\ 1\leq i\neq j\leq r\Bigr],
\end{equation}
whose elements are the power series in the \(X_i\) and \(Y_i\) in which poles along the hyperplanes \(X_i=0\) and \(X_i=X_j\) are allowed.  Since the inverted elements are non-zerodivisors, the power-series ring embeds into \(\cL_r(R)\).  We write \(\BARI^{\mathrm{loc}}\) for the space of bimoulds with \(A_0=0\) whose depth-\(r\) component lies in \(\cL_r(R)\), and \(\BIMU^{\mathrm{loc}}=\one+\BARI^{\mathrm{loc}}\) for the normalized bimoulds with these poles.  All substitutions occurring in the flexion formulas of \Crefrange{sec:flexion}{sec:polynomial} preserve \(\cL_r(R)\).  The swap turns a pole in the \(X_i\) into a pole in the \(Y_i\), and for this reason more poles have to be allowed in \Cref{sec:swap-problem}.

For a bimould \(M\), a linear form \(t\) in the upper variables, and \(1\leq p\leq q\), put
\[
 M^{(t)}\langle p,q\rangle
 :=M_{q-p+1}\bi{X_p-t,\ldots,X_q-t}{Y_p,\ldots,Y_q},
 \qquad M^{(t)}\langle p,p-1\rangle:=1,
\]
and write \(Y_{p:q}=Y_p+\cdots+Y_q\).  We also use \(w_i=\bi{X_i}{Y_i}\).

The concatenation product is
\begin{equation}\label{eq:mu}
 \muop(M,N)(w_1\cdots w_n)
 =\sum_{p=0}^nM(w_1\cdots w_p)N(w_{p+1}\cdots w_n).
\end{equation}
Its unit is \(\one\), and \(\invmu(M)\) denotes the inverse of a normalized bimould.  For an iterated product we write
\[
 \muop(M_1,\ldots,M_s)
 =\muop\bigl(\muop(M_1,\ldots,M_{s-1}),M_s\bigr).
\]

For normalized bimoulds \(B,C\), define \(\gaxit_{B,C}(M)(\varnothing)=M(\varnothing)\) and, for \(n\geq1\),
\begin{align}\label{eq:gaxit}
 \gaxit_{B,C}(M)(w_1\cdots w_n)
 ={}&\sum_{s\geq1}
 \sum_{\substack{0=q_0<p_1\leq q_1<\cdots<p_s\leq q_s=n}}
 M(\widehat w_{p_1}\cdots\widehat w_{p_s})\notag\\
 &\mathrel{\phantom{=}}\prod_{i=1}^s
 B^{(X_{p_i})}\langle q_{i-1}+1,p_i-1\rangle
 C^{(X_{p_i})}\langle p_i+1,q_i\rangle,
\end{align}
where
\[
 \widehat w_{p_i}=\bi{X_{p_i}}{Y_{q_{i-1}+1:q_i}}.
\]
Thus a term selects marked letters.  Each marked letter absorbs the lower variables in its block, while the upper variables in the two adjacent blocks are translated by its upper variable.  This is \'{E}calle's standard \(\gaxit\).

We use the multiplicativity of \(\gaxit\) for \(\muop\) \cite{Kaw1}*{Proposition~3.12}.  Its composition law \cite{Kaw1}*{Proposition~3.10} is
\begin{equation}\label{eq:gaxi}
 \gaxit_{B,C}\circ\gaxit_{A,D}
 =\gaxit_{E,F},
 \qquad
 (E,F)=
 \bigl(\muop(\gaxit_{B,C}(A),B),
       \muop(C,\gaxit_{B,C}(D))\bigr).
\end{equation}
This defines the associative GAXI product \(\gaxi((A,D),(B,C))=(E,F)\) on normalized pairs.

Put
\[
 \ganit_C=\gaxit_{\one,C},\qquad
 \garit_Q=\gaxit_{Q,\invmu(Q)},\qquad
 \gari(P,Q)=\muop(\garit_Q(P),Q).
\]
The pairs \((Q,\invmu(Q))\) form a subgroup of GAXI.  In particular,
\[
 \garit_Q\circ\garit_P=\garit_{\gari(P,Q)},
\]
and GARI is an associative group law on \(\BIMU^{\mathrm{loc}}\).

Over the dual numbers, define \(\arit\) by
\[
 \garit_{\one+\eps b}
 =\operatorname{id}+\eps\arit(b).
\]
For \(a,b\in\BARI^{\mathrm{loc}}\), define
\begin{align}\label{eq:ari}
 \preari(a,b)&=\arit(b)(a)+\muop(a,b),\notag\\
 \ari(a,b)&=\preari(a,b)-\preari(b,a).
\end{align}
Equivalently,
\begin{equation}\label{eq:gari-tangent}
 \ari(a,b)=[\eps\eta]\bigl(
 \gari(\one+\eps a,\one+\eta b)
 -\gari(\one+\eta b,\one+\eps a)\bigr),
\end{equation}
where $[\eps\eta]$ denotes the coefficient at $\eps\eta$.  In particular, ari is a Lie bracket.  It preserves alternal bimoulds by \cite{Sc}*{Proposition~2.5.2}.

The special normalized bimoulds pic and poc are
\begin{equation}\label{eq:pic-poc}
 \pic_r=\frac{1}{X_1\cdots X_r},\qquad
 \poc_r=-\frac{1}{X_1(X_1-X_2)\cdots(X_{r-1}-X_r)}
 \qquad(r\geq1),
\end{equation}
with depth-zero component one.  We put
\begin{equation}\label{eq:T}
 T:=\ganit_{\pic},\qquad T^{-1}=\ganit_{\poc},
\end{equation}
the second identity being the flexion-unit inverse formula \cite{Kaw1}*{Proposition~5.3}.  Moreover, \(T\) restricts to an isomorphism from alternal to alternil bimoulds in \(\BARI^{\mathrm{loc}}\), with inverse \(T^{-1}\).  This transport was first proved by Komiyama \cite{Ko}*{Theorem~3.24} in the mould setting.  The bimould version used here is \cite{Bu1}*{Proposition~5.9}.

\begin{dfn}\label{def:uri}
For \(a,b\in\BARI^{\mathrm{loc}}\), set
\begin{equation}\label{eq:uri}
 \uri(a,b)=T\bigl(\ari(T^{-1}a,T^{-1}b)\bigr).
\end{equation}
\end{dfn}

The bracket uri is the ari bracket transported by \(T\).  Therefore it is a Lie bracket on \(\BARI^{\mathrm{loc}}\) which preserves alternility.

\section{A regular power-series operation}\label{sec:regular}

We define a power-series operation whose only quotients are divided differences.  The same construction appears in the joint work of the first author and Burmester on the coproduct for formal multiple Eisenstein series \cite{BB2}.

\begin{dfn}\label{def:tilde}
For a normalized bimould \(A\), define \(\widetilde A_0=1\) and, for \(r\geq1\),
\begin{align}\label{eq:tilde-recursion}
 \widetilde A_r={}&-A_r
 -\sum_{i=1}^{r-1}A^{(0)}\langle1,i\rangle
 \widetilde A^{(0)}\langle i+1,r\rangle\notag\\
 &+\sum_{i=1}^r
 \frac{A^{(0)}\langle1,i-1\rangle
 \widetilde A^{(0)}\langle i+1,r\rangle
 -A^{(X_i)}\langle1,i-1\rangle
 \widetilde A^{(X_i)}\langle i+1,r\rangle}{X_i}.
\end{align}
\end{dfn}

Since the right-hand side involves only smaller depths, \eqref{eq:tilde-recursion} is a recursion on depth.

\begin{thm}\label{thm:tilde-regular}
The recursion \eqref{eq:tilde-recursion} defines a unique normalized power-series bimould.  If \(A\) is polynomial, then \(\widetilde A\) is polynomial.
\end{thm}

\begin{proof}
For fixed \(i\), put
\[
 F_i(t)=A^{(t)}\langle1,i-1\rangle
 \widetilde A^{(t)}\langle i+1,r\rangle.
\]
The two blocks omit the index \(i\), so \(X_i\) enters \(F_i(t)\) only through \(t\), and therefore \(F_i(0)-F_i(X_i)\) is divisible by \(X_i\).  Every quotient in \eqref{eq:tilde-recursion} is therefore a formal power series.  Induction on the depth proves existence and uniqueness.  The same argument in the polynomial ring proves polynomiality.
\end{proof}

\begin{dfn}\label{def:diamond}
For normalized power-series bimoulds \(A,B\), define
\begin{align}\label{eq:diamond-direct}
 (A\diamond B)_r
 \bi{X_1,\ldots,X_r}{Y_1,\ldots,Y_r}
 :={}&\sum_{\substack{0\leq j\leq r\\
 0=r_0<n_1\leq r_1<\cdots<n_j\leq r_j\leq r}}
 Q^A_{\boldsymbol n,\boldsymbol r}
 B_j\bi{X_{n_1},\ldots,X_{n_j}}
 {Y_{1:r_1},Y_{r_1+1:r_2},\ldots,Y_{r_{j-1}+1:r_j}},
\end{align}
where
\begin{equation}\label{eq:diamond-left-factor}
 Q^A_{\boldsymbol n,\boldsymbol r}
 =\left(\prod_{i=1}^j
 A^{(X_{n_i})}\langle r_{i-1}+1,n_i-1\rangle
 \widetilde A^{(X_{n_i})}\langle n_i+1,r_i\rangle\right)
 A^{(0)}\langle r_j+1,r\rangle.
\end{equation}
For \(j=0\), the displayed lists and product are empty and \(r_j=r_0=0\).
\end{dfn}

The sum is finite in every depth.  The formula and \cref{thm:tilde-regular} show that \(A\diamond B\) is a power-series bimould and is componentwise polynomial when \(A\) and \(B\) are polynomial.  The following marked-letter form will also be useful.

\begin{lem}\label{lem:translation}
For normalized power-series bimoulds,
\begin{equation}\label{eq:diamond-gaxit}
 A\diamond B=\muop\bigl(\gaxit_{A,\widetilde A}(B),A\bigr).
\end{equation}
\end{lem}

\begin{proof}
Expand the right-hand side using \eqref{eq:mu} and \eqref{eq:gaxit}.  Write \(n_i=p_i\) and \(r_i=q_i\) for the marked positions and their endpoints.  The final split in \(\muop\) gives the factor \(A^{(0)}\langle r_j+1,r\rangle\).  The remaining factors and lower arguments are exactly those in \eqref{eq:diamond-direct} and \eqref{eq:diamond-left-factor}.
\end{proof}

Let \(V\) be concentrated in depth one with \(V_1\bi{X_1}{Y_1}=1/X_1\), and put \(\jmath=\one-V\).  For \(A\in\BIMU^{\mathrm{loc}}\), the same triangular recursion \eqref{eq:tilde-recursion} defines \(\widetilde A\), since every \(X_i\) is invertible in \(\cL_r(R)\).

\begin{lem}\label{lem:tilde-graph}
Let \(A\) be normalized, either as a power-series bimould or as a localized bimould.  A bimould \(K\in\BIMU^{\mathrm{loc}}\) satisfies
\begin{equation}\label{eq:tilde-graph}
 \gaxit_{A,K}(\jmath)=\muop(A,\jmath,K)
\end{equation}
if and only if \(K=\widetilde A\).
\end{lem}

\begin{proof}
Only decompositions with one marked letter contribute to \(\gaxit_{A,K}(V)\), and
\[
 \gaxit_{A,K}(V)(w_1\cdots w_n)
 =\sum_{i=1}^n
 \frac{A^{(X_i)}\langle1,i-1\rangle
 K^{(X_i)}\langle i+1,n\rangle}{X_i}.
\]
On the other hand,
\[
 \muop(A,\jmath,K)(w_1\cdots w_n)
 =\sum_{i=0}^nA^{(0)}\langle1,i\rangle
 K^{(0)}\langle i+1,n\rangle
 -\sum_{i=1}^nA^{(0)}\langle1,i-1\rangle
 \frac{K^{(0)}\langle i+1,n\rangle}{X_i}.
\]
Substituting these formulas in \eqref{eq:tilde-graph} and solving for \(K_n\) gives \eqref{eq:tilde-recursion}.  Uniqueness finishes the proof.
\end{proof}

\begin{thm}\label{thm:diamond-associative}
The operation \(\diamond\) is associative on normalized power-series bimoulds and has unit \(\one\).
\end{thm}

\begin{proof}
The unit follows directly from \eqref{eq:diamond-direct}.  Put \(\Gamma(A)=(A,\widetilde A)\).  Multiplicativity and \eqref{eq:gaxi}, together with \cref{lem:translation,lem:tilde-graph}, give
\begin{equation}\label{eq:graph-product}
 \gaxi(\Gamma(A),\Gamma(B))=\Gamma(B\diamond A).
\end{equation}
To see this, we apply the operator corresponding to the left-hand pair to \(\jmath\) and get
\begin{align*}
 \gaxit_{B,\widetilde B}
 \bigl(\gaxit_{A,\widetilde A}(\jmath)\bigr)
 &=\gaxit_{B,\widetilde B}\bigl(\muop(A,\jmath,\widetilde A)\bigr)\\
 &=\muop\bigl(\gaxit_{B,\widetilde B}(A),
 \muop(B,\jmath,\widetilde B),
 \gaxit_{B,\widetilde B}(\widetilde A)\bigr).
\end{align*}
Write the GAXI product as \((C,L)\).  Its first component is \(C=B\diamond A\), and the last display is \(\muop(C,\jmath,L)\).  By \cref{lem:tilde-graph}, we have \(L=\widetilde C\), which proves \eqref{eq:graph-product}.  Associativity now follows from associativity of GAXI and injectivity of \(\Gamma\).
\end{proof}

In depth \(r\), the direct formula has the form
\[
 (A\diamond B)_r=A_r+B_r
 +\text{terms involving only components of depths below }r.
\]
Hence \(A\diamond B=\one\) determines a right inverse \(B\) recursively.  If \(B\diamond C=\one\), then associativity gives \(A=A\diamond(B\diamond C)=(A\diamond B)\diamond C=C\).  This gives \(B\diamond A=\one\), and \(\diamond\) is a group law.  The depth-one formula is
\begin{equation}\label{eq:diamond-depth-one}
 (A\diamond B)_1=A_1+B_1.
\end{equation}

\section{Conjugation with GARI}\label{sec:conjugation}

Besides \eqref{eq:T} we will use the identity
\begin{equation}\label{eq:flexion-unit-identities}
 T(\poc)=\invmu(\pic)=\jmath.
\end{equation}
To see this, apply \eqref{eq:gaxi} to \(\ganit_{\pic}\circ\ganit_{\poc}=\gaxit_{\one,\pic}\circ\gaxit_{\one,\poc}\).  The resulting pair is \(\bigl(\one,\muop(\pic,T(\poc))\bigr)\), so that \(\ganit_{\pic}\circ\ganit_{\poc}=\ganit_{\muop(\pic,T(\poc))}\).  By \eqref{eq:T} this composition is the identity, and a bimould \(F\) with \(\ganit_F=\operatorname{id}\) satisfies \(F=\one\), as one sees by applying \(\ganit_F\) to a bimould concentrated in depth one.  We obtain \(\muop(\pic,T(\poc))=\one\), which gives \(T(\poc)=\invmu(\pic)\).  A direct computation with \eqref{eq:pic-poc} gives \(\muop(\jmath,\pic)=\one\), and therefore \(\invmu(\pic)=\jmath\).

\begin{prop}\label{prop:diamond-conjugation}
For a normalized \(A\) with poles, put \(Q=T^{-1}(A)\).  Then
\begin{equation}\label{eq:garit-conjugation}
 T\circ\garit_Q\circ T^{-1}=\gaxit_{A,\widetilde A}.
\end{equation}
Therefore, for normalized power-series bimoulds \(A,B\),
\begin{equation}\label{eq:diamond-gari}
 A\diamond B=T\bigl(\gari(T^{-1}B,T^{-1}A)\bigr).
\end{equation}
The order on the right is reversed.
\end{prop}

\begin{proof}
The GAXI composition law shows that the operator on the left of \eqref{eq:garit-conjugation} is \(\gaxit_{A,K}\), where
\[
 K=\muop\bigl(\pic,\invmu(A),T(\garit_Q(\poc))\bigr).
\]
Using \eqref{eq:flexion-unit-identities}, we obtain
\begin{align*}
 \gaxit_{A,K}(\jmath)
 &=T\bigl(\garit_Q(T^{-1}(\jmath))\bigr)
 =T\bigl(\garit_Q(\poc)\bigr)\\
 &=\muop\bigl(A,\invmu(A),T(\garit_Q(\poc))\bigr)
 =\muop(A,\jmath,K).
\end{align*}
The last equality uses \(\muop(\jmath,\pic)=\one\).  By \cref{lem:tilde-graph}, we have \(K=\widetilde A\), which proves \eqref{eq:garit-conjugation}.  Together with \cref{lem:translation}, this gives
\begin{align*}
 A\diamond B
 &=\muop\bigl(T(\garit_Q(T^{-1}B)),T(Q)\bigr)\\
 &=T\bigl(\muop(\garit_Q(T^{-1}B),Q)\bigr),
\end{align*}
which is \eqref{eq:diamond-gari}.
\end{proof}

The conjugation identity also gives the following result, which is used in \cite{BB2}.

\begin{thm}\label{thm:diamond-symmetril}
If \(A\) and \(B\) are normalized symmetril bimoulds, then \(A\diamond B\) is symmetril.
\end{thm}

\begin{proof}
The symmetral--symmetril transport for \(T\) and its inverse is \cite{Ko}*{Theorem~3.24}.  Normalized symmetral bimoulds form a GARI subgroup by \cite{Ko}*{Theorem~A.7}.  The proofs of both results use only finitely many shuffle and flexion identities in each depth.  These identities remain valid after base change to the localization \eqref{eq:poles}.  Therefore, the right-hand side of \eqref{eq:diamond-gari} is symmetril.  The formula \eqref{eq:diamond-direct} shows that if $A$ and $B$ are power-series bimoulds, then \(A\diamond B\) is a power-series bimould as well.
\end{proof}

\section{Polynomial closure}\label{sec:polynomial}

The definition \eqref{eq:uri} of the uri bracket contains poles, even for polynomial inputs.  The direct formula for \(\diamond\) removes them.  We put
\[
 \BARIpol
 :=\{a\in\BARI\mid a_r\text{ is polynomial for every }r
 \text{ and }a_r=0\text{ for all but finitely many }r\},
\]
and we write \(\mathfrak B\) for the space of elements of \(\BARIswapil\) which lie in \(\BARIpol\), with \(R=\QQ\).

\begin{thm}\label{thm:regularity}
The uri bracket maps power-series bimoulds to power-series bimoulds in every depth and componentwise polynomial bimoulds to componentwise polynomial bimoulds, and it preserves \(\BARIpol\).
\end{thm}

\begin{proof}
Let
\[
 S=R[\eps,\eta]/(\eps^2,\eta^2),\qquad
 A=\one+\eps a,\qquad B=\one+\eta b.
\]
The operation \(\diamond\) is defined over \(S\).  The mixed coefficients of \(A\diamond B\) and \(B\diamond A\) are power series, and they are polynomials whenever the components of \(a\) and \(b\) are polynomials.

Equation \eqref{eq:diamond-gari} and the definition \eqref{eq:gari-tangent} give
\begin{equation}\label{eq:tangent-uri}
 [\eps\eta]\bigl(A\diamond B-B\diamond A\bigr)
 =-T\bigl(\ari(T^{-1}a,T^{-1}b)\bigr)
 =-\uri(a,b).
\end{equation}
This proves power-series and componentwise polynomial closure.

For the last statement, give a monomial in depth \(r\) and ordinary degree \(d\) the weight \(r+d\).  All substitutions occurring in \eqref{eq:tilde-recursion} and \eqref{eq:diamond-direct} preserve the ordinary degree, products add depths and degrees, and each divided difference raises the depth by one while lowering the degree by one.  The weight is therefore preserved, and the mixed coefficient has weight equal to the sum of the input weights.  In weight \(k\), a depth- \(r\) component has ordinary degree \(k-r\) and hence vanishes for \(r>k\).  Every finite-depth polynomial bimould is a finite sum of weight-homogeneous bimoulds.  Its bracket therefore has finite depth support.
\end{proof}

\section{Swap invariance}\label{sec:swap-problem}

By \cref{thm:regularity}, it remains to prove swap invariance.

\begin{thm}\label{thm:swap}
If \(a,b\in\BARIswapil\), then
\[
 \swapop\bigl(\uri(a,b)\bigr)=\uri(a,b).
\]
\end{thm}

Burmester constructed a uri Lie subalgebra of the polynomial elements of \(\BARIswapil\) in \cite{Bu1}*{Theorem~5.20}.  Our proof uses the dimorphic transports of the second author \cite{Kaw2} and the bialternal identities given in \cite{Sc}.

\subsection{A fixed-point formulation}

The symmetric flexion unit
\[
 E(u,v)=\frac1u+\frac1v
\]
has product bimould
\[
 \cP_r\bi{X_1,\ldots,X_r}{Y_1,\ldots,Y_r}
 =\prod_{i=1}^r\left(\frac1{Y_i}+\frac1{X_i}\right),
 \qquad \cP_0=1.
\]
The unit \(E\) is not the standard flexion unit \(1/X\) defining symmetrility and alternility.  In \'{E}calle's notation, \(E\) is the polar unit~\(\operatorname{Pai}_{1,1}\) and \(\cP\) is the bimould \(\operatorname{paic}_{1,1}\) \cite{Ec2}*{Sections~3.2 and~3.7}.  Since \(\cP\) has poles along \(Y_i=0\), we work in the larger ring
\[
 \widehat{\cL}_r(R)
 :=R[[X_1,\ldots,X_r,Y_1,\ldots,Y_r]]
   \Bigl[\tfrac1\ell\ \Big|\ \ell\text{ a nonzero }\QQ\text{-linear form
   in the }X_i,Y_i\Bigr],
\]
in which poles along every hyperplane through the origin are allowed.  It contains \(\cL_r(R)\) and \(\cP\).  All bimoulds in this section have their depth-\(r\) components in \(\widehat{\cL}_r(R)\).

For a bimould \(A\), define
\begin{align*}
 \antiop(A)(w_1\cdots w_r)&=A(w_r\cdots w_1),\\
 \negop(A)_r\bi{X_1,\ldots,X_r}{Y_1,\ldots,Y_r}
 &=A_r\bi{-X_1,\ldots,-X_r}{-Y_1,\ldots,-Y_r}.
\end{align*}

Now define
\begin{equation}\label{eq:fixed-map}
 \mathscr J(Q)=\ganit_{\cP}^{-1}\bigl(\swapop(Q)\bigr).
\end{equation}
Then swap invariance becomes the following fixed-point condition.

\begin{lem}\label{lem:swap-ganit}
Let \(B\) be alternil and put \(Q=T^{-1}(B)\).  Then
\begin{equation}\label{eq:swap-ganit}
 \swapop(B)=B
 \quad\Longleftrightarrow\quad
 \swapop(Q)=\ganit_{\cP}(Q)
 \quad\Longleftrightarrow\quad
 \mathscr J(Q)=Q.
\end{equation}
\end{lem}

\begin{proof}
Write \(\pac\) for \'{E}calle's normalized bimould with \(\pac_r=1/(Y_1\cdots Y_r)\) for \(r\geq1\) \cite{Ec2}*{Section~3.7}, and put \(\asna(C)=\antiop(\swapop(\negop(\antiop(C))))\).  In terms of the push operator \eqref{eq:push} below, we have \(\asna=\pushop\circ\swapop\) by \cite{Kaw1}*{Proposition~2.5}.  Reversing a marked decomposition in \eqref{eq:gaxit} interchanges the upper and the lower flexions in every block, which gives the standard identity (see also \cite{Ma}*{Proposition~3})
\begin{equation}\label{eq:swap-ganit-general}
 \swapop\circ\ganit_C=\ganit_{\asna(C)}\circ\swapop.
\end{equation}
From the explicit formula \eqref{eq:pic-poc} we get
\[
 \bigl(\negop\antiop(\poc)\bigr)_r
 =\frac{(-1)^{r+1}}{X_r(X_r-X_{r-1})\cdots(X_2-X_1)},
\]
and applying swap and then anti to this yields \(\asna(\poc)=\pac\).  We claim that
\begin{equation}\label{eq:P-factorization}
 \ganit_{\cP}=\ganit_{\pac}\circ T.
\end{equation}
Namely, by \eqref{eq:gaxi} the right-hand side is \(\ganit_F\), where
\[
 F=\muop\bigl(\pac,\ganit_{\pac}(\pic)\bigr).
\]
In \(\ganit_{\pac}(\pic)_r\) the first positions of the blocks form a subset \(S\subseteq\{1,\ldots,r\}\) containing \(1\), which contributes \(\prod_{i\in S}X_i^{-1}\), while the remaining positions contribute \(\prod_{i\notin S}Y_i^{-1}\).  In the product defining \(F\) the initial \(\pac\)-factor accounts for the positions before the least element of \(S\), so that every subset of \(\{1,\ldots,r\}\) occurs exactly once and
\[
 F_r=\sum_{S\subseteq\{1,\ldots,r\}}
 \prod_{i\in S}\frac1{X_i}\prod_{i\notin S}\frac1{Y_i}
 =\prod_{i=1}^r\left(\frac1{X_i}+\frac1{Y_i}\right)=\cP_r.
\]
Using \eqref{eq:swap-ganit-general} with \(C=\poc\) and then \eqref{eq:P-factorization}, we obtain
\[
 \swapop(Q)=\ganit_{\pac}(\swapop(B)),
 \qquad
 \ganit_{\cP}(Q)=\ganit_{\pac}(B).
\]
The first equivalence in \eqref{eq:swap-ganit} follows, since \(\ganit_{\pac}\) is invertible, and the second one is immediate from \eqref{eq:fixed-map}.
\end{proof}

Equations \eqref{eq:swap-ganit-general} and \eqref{eq:P-factorization} also give
\[
 \mathscr J=T^{-1}\circ\swapop\circ T,
 \qquad \mathscr J^2=\operatorname{id}.
\]
Namely, \eqref{eq:swap-ganit-general} gives \(\ganit_{\pac}^{-1}\circ\swapop=\swapop\circ T\), and \(\ganit_{\cP}=\ganit_{\pac}\circ T\) is \eqref{eq:P-factorization}.

We call a bimould \(H\) \emph{\(E\)-alternal} if \(\ganit_{\cP}^{-1}(H)\) is alternal, and we denote by \(\cD_E\) the space of all alternal bimoulds \(Q\) for which \(\swapop(Q)\) is \(E\)-alternal and which are even in depth one.  Put
\[
 \cD_E^{\mathrm{reg}}
 =\{Q\in\cD_E\mid T(Q)\text{ is a power-series bimould}\}.
\]
By \cref{lem:swap-ganit}, we have
\begin{equation}\label{eq:fixed-locus}
 T^{-1}(\BARIswapil)
 =\{Q\in\cD_E^{\mathrm{reg}}\mid\mathscr J(Q)=Q\}.
\end{equation}
For the inclusion from left to right, \cref{lem:swap-ganit} gives the fixed-point relation and shows that \(\swapop(Q)\) is \(E\)-alternal.  Conversely, let \(Q\in\cD_E^{\mathrm{reg}}\) be fixed by \(\mathscr J\).  Since \(T(Q)\) is a power-series bimould, \(Q=T^{-1}(T(Q))\) lies in \(\BARI^{\mathrm{loc}}\).  The transport by \(T\) shows that \(T(Q)\) is alternil, and \cref{lem:swap-ganit} shows that it is swap invariant.  The map \(T\) is the identity in depth one, so \(T(Q)\) also has the required parity.  This proves the converse inclusion.

By \eqref{eq:fixed-locus}, it remains to show that \(\cD_E\) is an ari Lie algebra and that \(\mathscr J\) is an automorphism.  Regularity will then follow from \cref{thm:regularity}.

\subsection{Dimorphic transport}

The push operator in our coordinates is
\begin{align}\label{eq:push}
 \pushop(A)_r\bi{X_1,\ldots,X_r}{Y_1,\ldots,Y_r}
 =A_r\bi{-X_r,X_1-X_r,\ldots,X_{r-1}-X_r}
 {-Y_1-\cdots-Y_r,Y_1,\ldots,Y_{r-1}}.
\end{align}
By \(\Lbial\) we denote the space of \emph{bialternal} bimoulds, which consists of all alternal bimoulds \(A\) for which \(\swapop(A)\) is alternal and \(A_1\) is even.  This is an ari Lie algebra \cite{Sc}*{Theorem~2.5.6}, and every element is push invariant \cite{Sc}*{Lemma~2.5.5}.  For push-invariant bimoulds \(A\) and \(B\), we also have
\begin{equation}\label{eq:swap-ari}
 \swapop\bigl(\ari(A,B)\bigr)
 =\ari\bigl(\swapop(A),\swapop(B)\bigr)
\end{equation}
by \cite{Sc}*{Lemma~2.4.1}.  The proofs of these three statements use finite shuffle sums and the substitutions listed in \cref{rem:specialization} below, so they remain valid over \(\widehat{\cL}\).  Since swap preserves \(\Lbial\), equation \eqref{eq:swap-ari} shows that it is an automorphism of this ari Lie algebra.

The function \(E(u,v)=1/u+1/v\) is odd and symmetric.  A direct calculation gives the tripartite identity
\[
 E(u_1,v_1)E(u_2,v_2)
 =E(u_1+u_2,v_1)E(u_2,v_2-v_1)
 +E(u_1+u_2,v_2)E(u_1,v_1-v_2).
\]
This shows that \(E\) is a flexion unit equal to its conjugate unit.

Writing \(E\) also for the bimould concentrated in depth one with that value, the associated primary bimould \(\invmu(\one-E)\) has components \(\prod_iE(u_i,v_i)\), so it is \(\cP\).  The space \(\cD_E\) is therefore the dimorphic space \(\ARI_{\underline{\mathrm{al}}/\underline{\mathfrak{ol}}}\) of \cite{Kaw2}*{Section~3}.  In its notation, \(O=E\), \(\mathfrak{ez}=\mathfrak{oz}=\cP\), and \(\mathfrak{ess}=\mathfrak{oss}\).  Hence
\[
 \ddot{\mathfrak o}\mathfrak{ss}
 =\ddot{\mathfrak e}\mathfrak{ss}
 =\swapop(\mathfrak{ess}).
\]
The secondary bimoulds are defined in \cite{Kaw1}*{Definition~7.1}.  For an invertible \(P\), write \(\adari(P)\) for the adjoint action induced by conjugation with \(P\) in GARI.  Explicitly,
\[
 \gari\bigl(\gari(P,\one+\eps A),\invgari(P)\bigr)
 =\one+\eps\adari(P)(A)
 \qquad\bmod \eps^2.
\]
The map \(\adari(P)\) is invertible and preserves the ari bracket.

The results of \cite{Kaw2} are formulated for decks, which allow substitution along every linear map between variable spaces \cite{Kaw1}*{Definitions~2.1 and 2.2}.  The ring \(\widehat{\cL}\) is not a deck, since a map can send an inverted form to zero.  For every commutative \(\QQ\)-algebra \(R\), however, the rings \(\widehat{\cL}_r(R)\), \(r\geq0\), form a family of functions in the sense of Furusho, Hirose and Komiyama \cite{FHK}*{Definition~1}.  This follows because an injective substitution sends every nonzero linear form to a nonzero linear form and therefore extends to the localization.

\begin{rem}\label{rem:specialization}
Notice that \'{E}calle's four flexions, the operators \(\swapop\), \(\pushop\), \(\antiop\), \(\negop\) and the shuffle permutations are injective substitutions, and in each fixed depth every operation of the flexion structure is a finite combination of these and of \(\muop\).  Therefore the constructions of \cites{Kaw1,Kaw2} make sense for bimoulds over a family of functions, and their depthwise identities continue to hold.  In particular, this applies over \(\widehat{\cL}\).  Representability of decks is used in \cites{Kaw1,Kaw2} only for the correspondence between pro-nilpotent Lie algebras and pro-unipotent groups.  Here the exponential and logarithm are finite in each depth.
\end{rem}

We recall the construction of the secondary bimoulds over \(\widehat{\cL}\).  In \cite{Kaw1}, the second author starts with
\[
 \mathfrak{re}_1=E,
 \qquad
 \mathfrak{re}_{r+1}=\arit(\mathfrak{re}_r)(E).
\]
Let \(e(x)=1-\exp(-x)\), and define the coefficients \(\epsilon_r\) by
\[
 \exp\left(\left(\sum_{r\geq1}\epsilon_rx^{r+1}\right)
 \frac{d}{dx}\right)(x)=e(x).
\]
Then
\begin{equation}\label{eq:secondary-explicit}
 \mathfrak{ess}
 =\expari\left(\sum_{r\geq1}\epsilon_r\mathfrak{re}_r\right).
\end{equation}
Here \(\expari\) is the depth-filtered Lie exponential for GARI.  This is the construction of \cite{Kaw1}*{Equations~(6.28), (6.29), Proposition~6.7 and Definition~6.11}.  The dotted secondary bimoulds are obtained by swap \cite{Kaw1}*{Definition~7.1}.  Each \(\mathfrak{re}_r\) is concentrated in depth \(r\), so \eqref{eq:secondary-explicit} is finite in every fixed depth.  By \cref{rem:specialization}, these bimoulds have components in \(\widehat{\cL}\), and the same applies to \(\invmu\), \(\gari\), \(\invgari\), \(\expari\), \(\logari\), and \(\adari\), whose defining expressions are finite in each depth.

\begin{lem}\label{lem:localized-secondary}
The bimoulds \(\mathfrak{ess}\) and \(\ddot{\mathfrak e}\mathfrak{ss}\) are symmetral.  Put
\[
 H_+=\adari(\mathfrak{ess}),
 \qquad
 H_-=\adari(\ddot{\mathfrak e}\mathfrak{ss}).
\]
Both maps preserve alternality and induce the identity on the associated graded for the depth filtration.  For every push-invariant bimould \(A\),
\begin{align}
 \swapop\bigl(H_+(A)\bigr)
 &=\ganit_{\cP}\bigl(H_-(\swapop(A))\bigr),
 \label{eq:second-fundamental}\\
 \swapop\bigl(H_-(A)\bigr)
 &=\ganit_{\cP}\bigl(H_+(\swapop(A))\bigr).\notag
\end{align}
Moreover, \(\ganit_{\cP}\) and its inverse induce the identity on the associated graded for the depth filtration.
\end{lem}

\begin{proof}
The symmetrality of the two secondary bimoulds is \cite{Kaw2}*{Theorem~A.7}.  The displayed identities are \cite{Kaw2}*{Corollary~3.2} specialized to \(O=E\), \(\mathfrak{oz}=\cP\), and \(\mathfrak{oss}=\mathfrak{ess}\).

By \cref{rem:specialization}, the first fundamental identity \cite{Sc}*{Corollary~2.8.6}, the crash identities \cite{Kaw1}*{Proposition~7.5 and Theorem~7.12}, secondary symmetrality \cite{Kaw2}*{Appendix~A}, and the shuffle-exponential argument of \cite{Ko}*{Theorem~A.7} remain valid over \(\widehat{\cL}\).

Conjugation by a symmetral element preserves the tangent space of this subgroup, so \(H_+\) and \(H_-\) preserve alternality.  The logarithms of the secondary bimoulds have positive depth, and every nonidentity term in their adjoint actions raises the depth.  The corresponding statement for \(\ganit_{\cP}^{\pm1}\) follows from the marked-letter formula.
\end{proof}

\begin{thm}\label{thm:transport}
Each map
\[
 H_\pm:\Lbial\longrightarrow\cD_E
\]
is an isomorphism of ari Lie algebras.  In particular, \(\cD_E\) is an ari Lie algebra.
\end{thm}

For bimoulds over a deck, this is the special case~\(O=E\) of \cite{Kaw2}*{Theorem~3.5}.  The proof below shows that the argument also applies over \(\widehat{\cL}\).

\begin{proof}
Let \(A\in\Lbial\).  Both \(A\) and \(\swapop(A)\) are alternal and push invariant.  By \cref{lem:localized-secondary}, \(H_+(A)\) is alternal, has the required depth-one parity and satisfies
\[
 \ganit_{\cP}^{-1}\bigl(\swapop(H_+(A))\bigr)
 =H_-(\swapop(A)).
\]
The right-hand side is alternal, and therefore \(H_+(A)\in\cD_E\).  The second formula in that lemma gives \(H_-(A)\in\cD_E\).  Both maps are injective because every adjoint action is invertible.

For a nonzero bimould \(B\), let \(\leadop(B)\) be its first nonzero depth component, viewed as a bimould concentrated in that depth.  If \(B\in\cD_E\), then \(\leadop(B)\) is alternal.  The bimould
\[
 \ganit_{\cP}^{-1}(\swapop(B))
\]
is also alternal.  Since \(\ganit_{\cP}^{-1}\) induces the identity on the associated graded, its first nonzero component is \(\swapop(\leadop(B))\).  In particular, \(\swapop(\leadop(B))\) is alternal, so \(\leadop(B)\) is bialternal.  It also has the required depth-one parity.  This is inherited from \(B_1\) when the leading depth is one, and otherwise its depth-one component is zero.  Hence
\begin{equation}\label{eq:lead-bialternal}
 \leadop(B)\in\Lbial.
\end{equation}

Starting with \(B_0=B\), define
\[
 C_n=\leadop(B_n),
 \qquad
 B_{n+1}=B_n-H_+(C_n),
\]
and stop if \(B_n=0\).  By \eqref{eq:lead-bialternal} and the inclusion proved in the first part, every \(B_n\) lies in \(\cD_E\).  Since \(H_+\) induces the identity on the associated graded, the lowest nonzero depth increases strictly at every step.  Therefore
\[
 C=\sum_{n\geq0}C_n
\]
is finite in each depth and belongs to \(\Lbial\).  Continuity in the depth filtration and telescoping give
\[
 H_+(C)=\sum_{n\geq0}(B_n-B_{n+1})=B.
\]
The same construction with \(H_-\) proves its surjectivity.  Since adjoint actions preserve ari and \(\Lbial\) is an ari Lie algebra, both maps are Lie algebra isomorphisms and \(\cD_E\) is an ari Lie algebra.
\end{proof}

\begin{prop}\label{prop:J-automorphism}
\(\mathscr J\) restricts to an automorphism of the ari Lie algebra \(\cD_E\), and
\begin{equation}\label{eq:J-conjugation}
 \mathscr J=H_-\circ\swapop\circ H_+^{-1}
 \qquad\text{on }\cD_E.
\end{equation}
\end{prop}

\begin{proof}
Let \(A\in\Lbial\).  By \cite{Sc}*{Lemma~2.5.5}, \(A\) is push invariant, so \eqref{eq:second-fundamental} applies and gives \(\mathscr J\bigl(H_+(A)\bigr) =\ganit_{\cP}^{-1}\bigl(\swapop(H_+(A))\bigr)=H_-(\swapop(A))\).  We therefore have \(\mathscr J\circ H_+=H_-\circ\swapop\) on \(\Lbial\), and \eqref{eq:J-conjugation} follows, since \(H_+\) is a bijection onto \(\cD_E\) by \cref{thm:transport}.  The right-hand side is an ari automorphism by \cref{thm:transport} and \eqref{eq:swap-ari}.
\end{proof}

\begin{proof}[Proof of \cref{thm:swap}]
Let \(a,b\in\BARIswapil\) and put \(Q_a=T^{-1}(a)\), \(Q_b=T^{-1}(b)\).  By \eqref{eq:fixed-locus}, both lie in \(\cD_E\) and are fixed by \(\mathscr J\).  Put \(N=\ari(Q_a,Q_b)\).  Since \(\cD_E\) is an ari Lie algebra, \(N\in\cD_E\), and \cref{prop:J-automorphism} gives
\[
 \mathscr J(N)
 =\ari\bigl(\mathscr J(Q_a),\mathscr J(Q_b)\bigr)
 =\ari(Q_a,Q_b)=N.
\]
By definition, \(T(N)=\uri(a,b)\), which is a power-series bimould by \cref{thm:regularity}, and componentwise polynomial for polynomial inputs.  We get \(N\in\cD_E^{\mathrm{reg}}\), and \eqref{eq:fixed-locus} gives \(\uri(a,b)=T(N)\in\BARIswapil\).
\end{proof}

\begin{proof}[Proof of \cref{thm:A}]
Let \(a,b\in\BARIswapil\).  The bracket \(\uri(a,b)\) is alternil, since uri preserves alternility (see \Cref{sec:flexion}), and it is a power-series bimould by \cref{thm:regularity}.  Its depth-one component vanishes by \eqref{eq:tangent-uri} and \eqref{eq:diamond-depth-one}, so the parity condition in \eqref{eq:swap-parity} holds, and it is swap invariant by \cref{thm:swap}.  This proves (i).  By \cref{thm:regularity}, the uri bracket preserves \(\BARIpol\), which together with (i) shows that \(\mathfrak B\) is a Lie subalgebra.  The proof of \cref{thm:regularity} also shows that the bracket of two weight-homogeneous elements has weight equal to the sum of their weights.  Therefore \(\mathfrak B\) is weight graded (see also \Cref{sec:examples}).
\end{proof}

\section{The balanced double shuffle Lie algebra}\label{sec:balanced}

We recall the double shuffle space \(\mathfrak{dm}_0\) and its balanced counterpart \(\mathfrak{bm}_0\), following \cite{Bu1}*{Sections~4.4, 5.4 and Appendix~B.2}.  Let \(\mathsf X=\{x_0,x_1\}\) and \(\mathsf Y=\{y_1,y_2,\ldots\}\).  On the completed duals of the shuffle and stuffle algebras, respectively, consider the coproducts
\[
 \Delta_{\!\shuffle}(x_i)=x_i\otimes1+1\otimes x_i,
 \qquad
 \Delta_*(y_n)=y_n\otimes1+1\otimes y_n
   +\sum_{j=1}^{n-1}y_j\otimes y_{n-j}.
\]
Let \(\Pi_Y:\QQ\langle\mathsf X\rangle\to \QQ\langle\mathsf Y\rangle\) kill words ending in \(x_0\) and send
\[
 x_0^{k_1-1}x_1\cdots x_0^{k_r-1}x_1
 \longmapsto y_{k_1}\cdots y_{k_r}.
\]
For \(\psi\in\QQ\langle\mathsf X\rangle\), put
\[
 \psi_*:=\Pi_Y(\psi)
 +\sum_{n\geq2}\frac{(-1)^{n-1}}{n}
       (\Pi_Y(\psi)\mid y_n)y_1^n.
\]
Then \(\mathfrak{dm}_0\) consists of the polynomials \(\psi\) such that
\[
 \begin{gathered}
  (\psi\mid x_0)=(\psi\mid x_1)=(\psi\mid x_0x_1)=0,\\
  \Delta_{\!\shuffle}(\psi)=\psi\otimes1+1\otimes\psi,
  \qquad
  \Delta_*(\psi_*)=\psi_*\otimes1+1\otimes\psi_*.
 \end{gathered}
\]
By Racinet \cite{Rac} (see also \cite{Bu1}*{Theorem~B.30}), the space \(\mathfrak{dm}_0\) is a weight-graded Lie algebra under the Ihara bracket
\[
 \{\psi_1,\psi_2\}
 =d_{\psi_1}(\psi_2)-d_{\psi_2}(\psi_1)+[\psi_1,\psi_2],
 \qquad
 d_\psi(x_0)=0,\quad d_\psi(x_1)=[x_1,\psi]\,.
\]

For the balanced version, let \(\mathcal B=\{b_0,b_1,b_2,\ldots\}\) and define
\[
 \Delta_{\mathrm b}(b_0)=b_0\otimes1+1\otimes b_0,
 \qquad
 \Delta_{\mathrm b}(b_n)=b_n\otimes1+1\otimes b_n
   +\sum_{j=1}^{n-1}b_j\otimes b_{n-j}\quad(n\geq1).
\]
Write \(\QQ\langle\mathcal B\rangle^0\) for the span of \(1\) and of the words not beginning in \(b_0\), and let \(\Pi_0:\QQ\langle\mathcal B\rangle\to \QQ\langle\mathcal B\rangle^0\) be the corresponding projection.  The involution \(\tau\) on this space is given by
\begin{equation}\label{eq:balanced-tau}
 \tau\bigl(b_{k_1}b_0^{m_1}\cdots b_{k_r}b_0^{m_r}\bigr)
 =b_{m_r+1}b_0^{k_r-1}\cdots b_{m_1+1}b_0^{k_1-1}.
\end{equation}
Burmester's space \(\mathfrak{bm}_0\) consists of the polynomials \(f\in\QQ\langle\mathcal B\rangle\) satisfying
\[
 \begin{gathered}
  (f\mid b_0)=0,
  \qquad \Delta_{\mathrm b}(f)=f\otimes1+1\otimes f,
  \qquad \tau(\Pi_0(f))=\Pi_0(f),\\
  (f\mid b_2)=(f\mid b_4)=(f\mid b_6)=0.
 \end{gathered}
\]
The weight of a word is the sum of its positive indices plus the number of occurrences of \(b_0\).  There is an explicit weight-preserving embedding of vector spaces
\[
 \theta:\mathfrak{dm}_0\lhook\joinrel\longrightarrow\mathfrak{bm}_0
\]
by \cite{Bu1}*{Theorem~4.28}. We next recall the bimould realization of \(\mathfrak{bm}_0\).  For \(f\in\QQ\langle\mathcal B\rangle\), set \(\rho_B(f)_0=0\) and, for \(r\geq1\), define
\begin{align*}
 \rho_B(f)_r\bi{X_1,\ldots,X_r}{Y_1,\ldots,Y_r}
 :={}&\sum_{\substack{k_1,\ldots,k_r\geq1\\m_1,\ldots,m_r\geq0}}
 (f\mid b_{k_1}b_0^{m_1}\cdots b_{k_r}b_0^{m_r})
 \prod_{i=1}^rX_i^{k_i-1}Y_i^{m_i},
\end{align*}
and put
\[
 (A^{\#_Y})_r\bi{X_1,\ldots,X_r}{Y_1,\ldots,Y_r}
 :=A_r\bi{X_1,\ldots,X_r}
 {Y_1,Y_1+Y_2,\ldots,Y_1+\cdots+Y_r}.
\]
In particular, \(\rho_B\) only sees \(\Pi_0(f)\).  Set
\[
 \Phi(f):=\rho_B(f)^{\#_Y}.
\]
By \cite{Bu1}*{Corollary~5.51}, restriction gives a weight-preserving vector-space isomorphism
\begin{equation}\label{eq:bm0-bari-isomorphism}
 \Phi:\mathfrak{bm}_0\overset{\sim}{\longrightarrow}\mathfrak B.
\end{equation}

For \(f\in\QQ\langle\mathcal B\rangle\), write \(d_f^q\) for the derivation defined in \cite{Bu1}*{Definition~3.12}, and let \(d_f^{\mathrm b}=-d_f^q\).  We set
\[
 \{f,g\}_{\mathrm b}
 :=d_f^{\mathrm b}(g)-d_g^{\mathrm b}(f)+[f,g]\,,
\]
which has the same shape as the Ihara bracket above.\footnote{This is the negative of the bracket \(d_f^q(g)-d_g^q(f)-[f,g]\) of \cite{Bu1}*{Definition~3.16}.  We choose the sign so that \eqref{eq:b-bracket-uri-comparison} holds without a sign.  The ari and uri brackets of \cite{Bu1} are the negatives of ours, since \cite{Bu1}*{Definition~C.22} uses the opposite order of the two preari arguments from \eqref{eq:ari}.}  The augmentation ideal
\[
 \mathfrak m_{\mathrm b}
 :=\{f\in\QQ\langle\mathcal B\rangle:(f\mid1)=0\}
\]
is a weight-graded Lie algebra for this bracket \cite{Bu1}*{Theorem~3.20}.  Notice that the isomorphism \eqref{eq:bm0-bari-isomorphism} alone does not prove closure of \(\mathfrak{bm}_0\), since the map \(\rho_B\) forgets every word beginning in \(b_0\).  We first show that \(\Pi_0\) determines the missing coefficients for elements of \(\ker(\partial_0)\).

Let \(v_0=b_0\), and define the primitive letters \(v_k\), \(k\geq1\), by
\[
 \sum_{k\geq1}v_kt^k
 =\log\left(1+\sum_{k\geq1}b_kt^k\right).
\]
The concatenation derivation of \cite{Bu1}*{Definition~4.46} is given by
\[
 \partial_0(v_0)=1,\qquad \partial_0(v_k)=0\quad(k\geq1).
\]
Since \(b_k\), \(k\geq1\), is a polynomial in \(v_1,v_2,\ldots\), this is equivalently
\[
 \partial_0(b_0)=1,\qquad \partial_0(b_k)=0\quad(k\geq1).
\]
The restriction of \(\Pi_0\) to \(\ker(\partial_0)\) is injective.  In fact, its inverse is the section
\[
 \operatorname{sec}_q(F)
 =\sum_{m\geq0}\frac{(-1)^m}{m!}v_0^m\partial_0^m(F)
 \qquad(F\in\QQ\langle\mathcal B\rangle^0)
\]
by \cite{Bu1}*{Proposition~4.47}.  The primitive elements for \(\Delta_{\mathrm b}\) form the free Lie algebra on the \(v_k\) \cite{Bu1}*{Proposition~4.32}.  Lazard elimination, together with \((f\mid b_0)=0\), therefore gives
\begin{equation}\label{eq:bm0-partial-kernel}
 \mathfrak{bm}_0\subseteq\ker(\partial_0).
\end{equation}

\begin{lem}\label{lem:partial-b-bracket}
The derivation \(\partial_0\) is a derivation for \(\{\ ,\ \}_{\mathrm b}\):
\begin{equation}\label{eq:partial-b-bracket}
 \partial_0\{f,g\}_{\mathrm b}
 =\{\partial_0f,g\}_{\mathrm b}+\{f,\partial_0g\}_{\mathrm b}.
\end{equation}
In particular, \(\ker(\partial_0)\) is closed under \(\{\ ,\ \}_{\mathrm b}\).
\end{lem}

\begin{proof}
We first show that
\begin{equation}\label{eq:partial-db}
 [\partial_0,d_f^{\mathrm b}]=d_{\partial_0f}^{\mathrm b}.
\end{equation}
It is enough to check this on the letters \(b_a\).  The index-raising maps \(\partial_i\), \(i\geq1\), of \cite{Bu1}*{Definition~3.8} leave every occurrence of \(b_0\) in place, and hence commute with \(\partial_0\).  The same is therefore true for the operators \(\delta_j\) built from them.  Likewise, the operators \(\mathrm{lc}_j^{(a,r)}\) of \cite{Bu1}*{Definition~3.10} only replace one positive-index letter by a string of positive-index letters, and so also commute with \(\partial_0\).  Applying \(\partial_0\) to the formula for \(d_w^{q}(b_a)\) in \cite{Bu1}*{Definition~3.12}, and using \(\partial_0(b_a)=0\), therefore gives
\[
 \partial_0\bigl(d_w^{q}(b_a)\bigr)
 =d_{\partial_0w}^{q}(b_a)\qquad(a\geq1).
\]
For a depth-zero word \(w=b_0^m\), this is the immediate identity \(\partial_0[b_0^m,b_a]=m[b_0^{m-1},b_a]\), using the separate depth-zero clause in \cite{Bu1}*{Definition~3.12}.  On \(b_0\), both sides vanish.  This proves \eqref{eq:partial-db} by linearity.  Applying it to the definition of \(\{f,g\}_{\mathrm b}\), and using that \(\partial_0\) is a derivation for the commutator, gives \eqref{eq:partial-b-bracket}.
\end{proof}

\begin{lem}\label{lem:b-bracket-uri-comparison}
For \(f,g\in\mathfrak{bm}_0\), one has
\begin{equation}\label{eq:b-bracket-uri-comparison}
 \Phi(\{f,g\}_{\mathrm b})=\uri(\Phi(f),\Phi(g)).
\end{equation}
\end{lem}

\begin{proof}
Put \(s_f(g)=d_f^{q}(g)+gf\), so that \(\{f,g\}_{\mathrm b}=s_g(f)-s_f(g)\).  With \(A=\Phi(f)\) and \(B=\Phi(g)\), \cite{Bu1}*{Theorem~5.52} gives
\begin{equation}\label{eq:b-preuri}
 \Phi(s_f(g))=\operatorname{preuri}_{\mathrm{Bu}}(A,B),
\end{equation}
where the right-hand side is the preuri operation of \cite{Bu1}*{Definition~5.15}.  We compare it with the linearization of \(\diamond\).  For a bimould \(F\), write
\begin{align*}
 F^{\mathrm t}
 &=F_{m-1}\bi{X_2,\ldots,X_m}{Y_2,\ldots,Y_m},&
 F^{\mathrm{t},1}
 &=F_{m-1}\bi{X_2-X_1,\ldots,X_m-X_1}{Y_2,\ldots,Y_m},\\
 F^{\mathrm h}
 &=F_{m-1}\bi{X_1,\ldots,X_{m-1}}{Y_1,\ldots,Y_{m-1}},&
 F^{\mathrm{h},m}
 &=F_{m-1}\bi{X_1-X_m,\ldots,X_{m-1}-X_m}
 {Y_1,\ldots,Y_{m-1}}.
\end{align*}
Put \(\lambda(A)=[\eps]\widetilde{\one+\eps A}\).  Linearizing \eqref{eq:tilde-recursion} gives
\begin{equation}\label{eq:lambda-preuri}
 \lambda(A)_0=0,\qquad \lambda(A)_1=-A_1,\qquad
 \lambda(A)_m=-A_m
 +\frac{\lambda(A)^{\mathrm t}-\lambda(A)^{\mathrm{t},1}}{X_1}
 +\frac{A^{\mathrm h}-A^{\mathrm{h},m}}{X_m}.
\end{equation}

For a polynomial \(F\) and nodes \(z_0,\ldots,z_r\), put
\[
 \mathcal D_r(F\mid z_0,\ldots,z_r)
 =\sum_{s=0}^r\frac{F(z_s)}{\prod_{k\ne s}(z_k-z_s)}.
\]
Define \(C(A)_0=C(A)_1=0\) and, for \(m\geq2\),
\[
 C(A)_m=\sum_{r=1}^{m-1}\bigl(L_{m,r}(A)-R_{m,r}(A)\bigr),
\]
where
\begin{align}
 L_{m,r}(A)
 &=\mathcal D_r\left(
 t\longmapsto
 A_{m-r}\bi{X_r-t,\ldots,X_{m-1}-t}
 {Y_r,\ldots,Y_{m-1}}
 \mathrel\big|0,X_1,\ldots,X_{r-1},X_m
 \right),\label{eq:Lmr-preuri}\\
 R_{m,r}(A)
 &=\mathcal D_r\left(
 t\longmapsto
 A_{m-r}\bi{X_{r+1}-t,\ldots,X_m-t}
 {Y_{r+1},\ldots,Y_m}
 \mathrel\big|0,X_1,\ldots,X_r
 \right).\label{eq:Rmr-preuri}
\end{align}
Using the marked-letter formula for \(\anit\), the two divided-difference sums in \cite{Bu1}*{Definition~5.12} can be written in flexion notation as
\begin{equation}\label{eq:urit-C-preuri}
 \operatorname{urit}_A=\arit(A)+\anit(C(A)).
\end{equation}
When a block of length \(r\) is inserted after a marked position \(i\), the operator \(\anit\) evaluates \(C(A)_r\) at
\[
 C(A)_r\bi{X_{i+1}-X_i,\ldots,X_{i+r}-X_i}
 {Y_{i+1},\ldots,Y_{i+r}}.
\]
Substitution of \eqref{eq:Lmr-preuri} and \eqref{eq:Rmr-preuri} gives, term by term, the two sums in that definition, with \(L_{m,r}\) and \(R_{m,r}\) corresponding to the two possible sides of the marked block.

The operator \(\mathcal D_r\) is symmetric in its nodes and satisfies
\begin{equation}\label{eq:divided-difference-preuri}
 \mathcal D_r(F\mid z_0,\ldots,z_r)
 =\frac{\mathcal D_{r-1}(F\mid z_1,\ldots,z_r)
 -\mathcal D_{r-1}(F\mid z_0,\ldots,z_{r-1})}{z_0-z_r}.
\end{equation}
For \(r=1\), definitions \eqref{eq:Lmr-preuri} and \eqref{eq:Rmr-preuri} give
\[
 L_{m,1}(A)=\frac{A^{\mathrm h}-A^{\mathrm{h},m}}{X_m},
 \qquad
 R_{m,1}(A)=\frac{A^{\mathrm t}-A^{\mathrm{t},1}}{X_1}.
\]
For \(r\geq2\), applying \eqref{eq:divided-difference-preuri} to the tail and translated-tail node sets gives
\[
 L_{m,r}(A)
 =\frac{L_{m-1,r-1}(A)^{\mathrm t}
       -L_{m-1,r-1}(A)^{\mathrm{t},1}}{X_1},
 \qquad
 R_{m,r}(A)
 =\frac{R_{m-1,r-1}(A)^{\mathrm t}
       -R_{m-1,r-1}(A)^{\mathrm{t},1}}{X_1}.
\]
Here the superscripts mean evaluation in the corresponding tail and translated-tail variables.  Summing these identities gives
\begin{equation}\label{eq:C-preuri}
 C(A)_1=0,\qquad
 C(A)_m
 =\frac{C(A)^{\mathrm t}-C(A)^{\mathrm{t},1}}{X_1}
 +\frac{A^{\mathrm{t},1}-A^{\mathrm t}}{X_1}
 +\frac{A^{\mathrm h}-A^{\mathrm{h},m}}{X_m}.
\end{equation}
Comparing \eqref{eq:lambda-preuri} and \eqref{eq:C-preuri} inductively gives
\begin{equation}\label{eq:C-lambda-preuri}
 C(A)=A+\lambda(A).
\end{equation}
Let \(\axit(P,Q)\) denote the linearization
\[
 \gaxit_{\one+\eps P,\one+\eps Q}
 =\operatorname{id}+\eps\axit(P,Q).
\]
Then \(\anit(Q)=\axit(0,Q)\) and \(\arit(P)=\axit(P,-P)\).  By \eqref{eq:urit-C-preuri} and \eqref{eq:C-lambda-preuri},
\[
 \operatorname{urit}_A
 =\axit(A,C(A)-A)=\axit(A,\lambda(A)).
\]
Linearizing \eqref{eq:diamond-gaxit} now gives
\begin{equation}\label{eq:diamond-preuri}
 [\eps\eta]\bigl((\one+\eps A)\diamond(\one+\eta B)\bigr)
 =\operatorname{preuri}_{\mathrm{Bu}}(A,B),
\end{equation}
since the \(\gaxit\) factor contributes \(\operatorname{urit}_A(B)\) and the last factor contributes \(\muop(B,A)\).  Antisymmetrizing \eqref{eq:diamond-preuri} and applying \eqref{eq:tangent-uri} yields
\[
 \operatorname{preuri}_{\mathrm{Bu}}(B,A)
 -\operatorname{preuri}_{\mathrm{Bu}}(A,B)=\uri(A,B).
\]
Since \(\{f,g\}_{\mathrm b}=s_g(f)-s_f(g)\), this together with \eqref{eq:b-preuri} proves \eqref{eq:b-bracket-uri-comparison}.
\end{proof}

\begin{proof}[Proof of \cref{thm:B}]
Let \(f,g\in\mathfrak{bm}_0\).  By \cref{thm:A}, the right-hand side of \eqref{eq:b-bracket-uri-comparison} lies in \(\mathfrak B\).  The isomorphism \eqref{eq:bm0-bari-isomorphism} therefore gives a unique \(h\in\mathfrak{bm}_0\) such that
\[
 \Phi(h)=\Phi(\{f,g\}_{\mathrm b}).
\]
Both \(h\) and \(\{f,g\}_{\mathrm b}\) lie in \(\ker(\partial_0)\) by \eqref{eq:bm0-partial-kernel} and \cref{lem:partial-b-bracket}.  Since \(\#_Y\) is invertible and \(\rho_B\) records every coefficient of \(\Pi_0\), the last display implies \(\Pi_0(h)=\Pi_0(\{f,g\}_{\mathrm b})\).  The injectivity of \(\Pi_0|_{\ker(\partial_0)}\) now gives \(h=\{f,g\}_{\mathrm b}\), proving closure.  The bracket already satisfies antisymmetry and the Jacobi identity on the ambient space by \cite{Bu1}*{Theorem~3.20}, and it preserves weight.  Therefore \(\mathfrak{bm}_0\) is a weight-graded Lie algebra.  Equation \eqref{eq:b-bracket-uri-comparison} shows that \(\Phi\) is a Lie algebra isomorphism.
\end{proof}

\begin{rem}\label{rem:hopf}
In \cite{BB2} the first author and Burmester construct a coproduct on \(\mathcal G^{\!f}\) modulo quasimodular forms, which turns this quotient into a connected graded Hopf algebra.  Conjecturally this Hopf algebra is \(\mathcal U(\mathfrak{bm}_0)^\vee\) from \eqref{eq:Gf-conjecture}.  It might be interesting to describe this coproduct directly on the space of bimoulds, where \cref{thm:A} now provides the Lie bracket, and to work out the whole Hopf algebra structure there.
\end{rem}

\section{Explicit elements in low weight}\label{sec:examples}

In this section we describe $\mathfrak B$ in low weight.  Recall that $\mathfrak B$ is the space of finite-depth polynomial elements of $\BARIswapil$, and that it is a Lie algebra under uri by \cref{thm:A}.  As in the proof of \cref{thm:regularity}, a monomial in depth $r$ of ordinary degree $d$ has weight $r+d$, and we write $\mathfrak B_k$ for the subspace of homogeneous elements of weight $k$.  Alternility, swap invariance and the depth-one parity are homogeneous conditions, so that
\[
 \mathfrak B=\bigoplus_{k\geq1}\mathfrak B_k,
 \qquad
 \uri(\mathfrak B_k,\mathfrak B_l)\subseteq\mathfrak B_{k+l}.
\]
For fixed $k$, the space $\mathfrak B_k$ is cut out by finitely many linear equations in the coefficients \eqref{eq:coefficient-expansion} of weight $k$, and can therefore be computed exactly.  For $k\leq9$ we obtain
\[
\begin{array}{c|ccccccccc}
 k & 1&2&3&4&5&6&7&8&9\\ \hline
 \dim\mathfrak B_k & 1&0&2&1&4&3&8&11&18
\end{array}
\]

\subsection*{The elements \texorpdfstring{$\xi_k^m$}{xi(k,m)}}

Let $E$ be the bimould with
\[
 E_1\bi{X_1}{Y_1}=-\frac{X_1+Y_1}2,
 \qquad
 E_2\bi{X_1,X_2}{Y_1,Y_2}=\frac14,
\]
and $E_r=0$ for $r\geq3$, and let $Q$ be the operator which multiplies the depth-$r$ component of a bimould by $X_1Y_1+\cdots+X_rY_r$.  In \cite{Ba} it is shown that
\begin{equation}\label{eq:delta}
 \delta(A)=T\bigl(Q(T^{-1}A)\bigr)+\uri(A,E)
\end{equation}
is a derivation of $(\mathfrak B,\uri)$ which raises the weight by two.  Since $T$ and $\uri$ do not change the depth-one component, $\delta$ is multiplication by $X_1Y_1$ in depth one.

The corresponding operator on $\mathfrak{bm}_0$ is
\[
 \delta_{\mathrm b}:=\Phi^{-1}\circ\delta\circ\Phi.
\]
By \cref{thm:B}, it is a derivation for $\{\ ,\ \}_{\mathrm b}$ and raises the weight by two.  For example,
\[
 \delta_{\mathrm b}(b_1)
 =b_2b_0-b_0b_2+b_1b_2-b_2b_1-b_1^2b_0+b_1b_0b_1.
\]
This is the element in \cite{Bu1}*{Example~4.53}.

Fix a normalized injection in \cite{Bu1}*{Equation~(B.32.1)}, and, for every odd $k\geq3$, denote the image of its weight-$k$ generator by $\sigma_k\in\mathfrak{dm}_0$, normalized by $(\sigma_k\mid x_0^{k-1}x_1)=1$.  Put
\begin{equation}\label{eq:xi-generators}
 f_1=b_1,\qquad f_k=\theta(\sigma_k)\quad(k\geq3\text{ odd}),
 \qquad \xi_k=\Phi(f_k)\quad(k\geq1\text{ odd}).
\end{equation}
Thus $(\xi_1)_1=1$ and $(\xi_1)_r=0$ for $r\ne1$.  The choices of $\sigma_k$, and hence of $\xi_k$, are noncanonical in general.  One may instead use the canonical zeta generators constructed by Dorigoni, Doroudiani, Drewitt, Hidding, Kleinschmidt, Schlotterer, Schneps and Verbeek \cite{DDDHKSSV}.  For $k=2n+1$ with $1\leq n\leq4$, an exact computation shows that $\xi_k$ is the unique element of $\mathfrak B_k$ of the form
\[
 \xi_k=M+\swapop(M),
\]
where the components of $M$ depend only on the lower variables $Y_1,\ldots,Y_r$ and $M_1=Y_1^{k-1}$.  In each case the depth-$k$ component of $\xi_k$ is the constant $1/k$.  Apart from its constant depth-$k$ component, $M$ is an alternal mould in the lower variables whose swap is alternil up to a constant, i.e.\ an element of weight $k$ of \'{E}calle's Lie algebra $\ARI_{\underline{\mathrm{al}}*\underline{\mathrm{il}}}$ \cite{Sc}.  For $m\geq0$ write
\[
 f_k^m:=\delta_{\mathrm b}^m(f_k),
 \qquad
 \xi_k^m:=\delta^m(\xi_k)=\Phi(f_k^m),
\]
and omit the superscript when $m=0$.  Notice that $\xi_k^m$ has weight $k+2m$.  Its depth-one component is $(X_1Y_1)^m(X_1^{k-1}+Y_1^{k-1})$ for $k\geq3$ and $(X_1Y_1)^m$ for $k=1$.  In low weight these elements generate $\mathfrak B$ (see \cref{prop:low-weight}).

\begin{prop}\label{prop:low-weight}
The following statements hold.
\begin{enumerate}[(i)]
\item The elements $\xi_k^m$, with $k\geq1$ odd, $m\geq0$, and $k+2m\leq9$, are well defined and lie in $\mathfrak B$.
\item For $k\leq9$, the space $\mathfrak B_k$ is spanned by the iterated uri brackets of the elements $\xi_\ell^m$ with $\ell+2m\leq k$.
\item For every $n\geq1$, one has
\begin{equation}\label{eq:eisenstein-relations}
 \uri(\xi_1,\xi_{2n+1})=0.
\end{equation}
Applying $\delta^m$ gives, for every $m\geq1$,
\begin{equation}\label{eq:differentiated-eisenstein-relations}
 \sum_{i=0}^m\binom mi
 \uri\bigl(\delta^i(\xi_1),\delta^{m-i}(\xi_{2n+1})\bigr)=0.
\end{equation}
Every linear relation among the brackets $\uri(\xi_k^r,\xi_\ell^s)$ of weight at most $8$ is a linear combination of the specializations of \eqref{eq:eisenstein-relations} and \eqref{eq:differentiated-eisenstein-relations} in those weights.
\end{enumerate}
\end{prop}

\begin{proof}
Part (i) follows from \eqref{eq:xi-generators} and the fact that $\delta$ preserves $\mathfrak B$.  Part (ii) and the completeness assertion in (iii) were verified by the exact computation described below.  We prove the general identity in (iii).  Fix an odd $k\geq3$ and put $h=\{f_k,f_1\}_{\mathrm b}$.  The construction of $\theta$ in \cite{Bu1}*{Theorem~4.28} splits the support of $f_k$ into words in $b_0,b_1$ and words containing no $b_0$.  The first part has zero bracket with $b_1$ by \cite{Bu1}*{Lemma~3.15(i)}, while the defining derivation formula preserves the subspace of words containing no $b_0$.  Hence $h$ contains no $b_0$.

By \cref{thm:B}, $h\in\mathfrak{bm}_0$, so it is primitive and its projection is $\tau$-invariant.  Since $h$ contains no $b_0$, $\Pi_0(h)=h$.  Formula \eqref{eq:balanced-tau} sends every word without $b_0$ other than a power of $b_1$ to a word containing $b_0$.  So a homogeneous $\tau$-invariant element of weight $k+1$ containing no $b_0$ is a multiple of $b_1^{k+1}$.  Since $b_1^{k+1}$ is not primitive, $h=0$.  This is the argument of \cite{Bu1}*{Proposition~4.56}, whose closure hypothesis is now provided by \cref{thm:B}.  Finally, \cref{lem:b-bracket-uri-comparison} gives \(\uri(\xi_1,\xi_k)=0\).  Equation \eqref{eq:differentiated-eisenstein-relations} follows because $\delta$ is a derivation of the uri bracket.
\end{proof}

The computations were done in SageMath with exact rational arithmetic.  The spaces $\mathfrak B_k$ were computed from the linear conditions on the coefficients, and the uri bracket was evaluated with the polynomial formula \eqref{eq:tangent-uri}.  The first term in \eqref{eq:delta} was evaluated with the contraction formula of \cite{Ba}, which also has no poles.  For weight at most five, both terms were compared with a direct evaluation of the localized formulas \eqref{eq:uri} and \eqref{eq:delta}.

The elements of weight at most five are listed below, written as $\xi=((\xi)_1,(\xi)_2,\ldots)$ with the depth components in order, together with the corresponding elements $f_k^m=\Phi^{-1}(\xi_k^m)$ of $\mathfrak{bm}_0$.  The latter were obtained as in the proof of \cref{thm:B}: the coefficients of the words not beginning in $b_0$ are read off from $\rho_B(f_k^m)$, and the remaining coefficients are given by the section $\operatorname{sec}_q$.  The elements of weight seven and nine are considerably longer, for example $\xi_7$ has $127$ and $\xi_3^2$ has $322$ nonzero coefficients.

\begingroup\footnotesize
\smallskip\noindent\textit{Weight $1$.} Here $\xi_1=(1,0,\dots)$ and $f_1=b_1$.
\par\smallskip
\noindent\textit{Weight $3$.}
\begin{align*}
 \xi_3 &= \bigl(X_1^2 +Y_1^2, \quad X_1 -2X_2 -Y_1 +Y_2, \quad \tfrac{1}{3}, 0, \dots\bigr),\\
 f_3 &= b_3 -2b_1b_2 +b_2b_1 +b_0^2b_1 -2b_0b_1b_0 +b_0b_1^2 +b_1b_0^2 -2b_1b_0b_1 +b_1^2b_0 +\tfrac{1}{3}b_1^3.\\[1ex]
 \xi_1^1 &= \bigl(X_1Y_1, \quad -X_1 +X_2 -Y_2, 0, \dots\bigr),\\
 f_1^1 &= -b_0b_2 +b_1b_2 +b_2b_0 -b_2b_1 +b_1b_0b_1 -b_1^2b_0.
\end{align*}
\noindent\textit{Weight $5$.}
\begin{align*}
 \xi_5 &= \bigl(X_1^4 +Y_1^4, \quad 2X_1^3 -\tfrac{11}{2}X_1^2X_2 +\tfrac{9}{2}X_1X_2^2 -3X_2^3 -2Y_1^3 -\tfrac{1}{2}Y_1^2Y_2 +\tfrac{1}{2}Y_1Y_2^2 +2Y_2^3,\\
 &\quad 2X_1^2 -\tfrac{11}{2}X_1X_2 +\tfrac{9}{2}X_1X_3 -\tfrac{1}{2}X_2^2 +2X_2X_3 -\tfrac{1}{2}X_3^2 +2Y_1^2 -\tfrac{3}{2}Y_1Y_2 +3Y_1Y_3 -4Y_2^2\\
 &\quad -\tfrac{3}{2}Y_2Y_3 +2Y_3^2, \quad X_1 -4X_2 +6X_3 -4X_4 -Y_1 +3Y_2 -3Y_3 +Y_4, \quad \tfrac{1}{5}, 0, \dots\bigr),\\
 f_5 &= b_5 -3b_1b_4 +\tfrac{9}{2}b_2b_3 -\tfrac{11}{2}b_3b_2 +2b_4b_1 -\tfrac{1}{2}b_1^2b_3 +2b_1b_2^2 -\tfrac{1}{2}b_1b_3b_1 +\tfrac{9}{2}b_2b_1b_2\\
 &\quad -\tfrac{11}{2}b_2^2b_1 +2b_3b_1^2 -4b_1^3b_2 +6b_1^2b_2b_1 -4b_1b_2b_1^2 +b_2b_1^3 +b_0^4b_1 -4b_0^3b_1b_0\\
 &\quad +2b_0^3b_1^2 +6b_0^2b_1b_0^2 -\tfrac{11}{2}b_0^2b_1b_0b_1 -\tfrac{1}{2}b_0^2b_1^2b_0 +2b_0^2b_1^3 -4b_0b_1b_0^3\\
 &\quad +\tfrac{9}{2}b_0b_1b_0^2b_1 +2b_0b_1b_0b_1b_0 -\tfrac{11}{2}b_0b_1b_0b_1^2 -\tfrac{1}{2}b_0b_1^2b_0^2 +\tfrac{9}{2}b_0b_1^2b_0b_1\\
 &\quad -3b_0b_1^3b_0 +b_0b_1^4 +b_1b_0^4 -3b_1b_0^3b_1 +\tfrac{9}{2}b_1b_0^2b_1b_0 -\tfrac{1}{2}b_1b_0^2b_1^2 -\tfrac{11}{2}b_1b_0b_1b_0^2\\
 &\quad +2b_1b_0b_1b_0b_1 +\tfrac{9}{2}b_1b_0b_1^2b_0 -4b_1b_0b_1^3 +2b_1^2b_0^3 -\tfrac{1}{2}b_1^2b_0^2b_1 -\tfrac{11}{2}b_1^2b_0b_1b_0\\
 &\quad +6b_1^2b_0b_1^2 +2b_1^3b_0^2 -4b_1^3b_0b_1 +b_1^4b_0 +\tfrac{1}{5}b_1^5.\\[1ex]
 \xi_3^1 &= \bigl(X_1^3Y_1 +X_1Y_1^3, \quad -X_1^3 +\tfrac{3}{2}X_1^2X_2 -\tfrac{3}{2}X_1X_2^2 +X_2^3 +X_1^2Y_1 -X_1^2Y_2 -2X_1X_2Y_1\\
 &\quad +X_1X_2Y_2 -2X_2^2Y_2 -2X_1Y_1^2 -X_1Y_1Y_2 -X_1Y_2^2 +X_2Y_1^2 +X_2Y_1Y_2 +2X_2Y_2^2 -\tfrac{3}{2}Y_1^2Y_2\\
 &\quad -\tfrac{3}{2}Y_1Y_2^2 -Y_2^3, \quad -X_1^2 +\tfrac{1}{2}X_1X_2 +\tfrac{3}{2}X_1X_3 +\tfrac{3}{2}X_2^2 -2X_2X_3 -\tfrac{1}{2}X_3^2 +\tfrac{4}{3}X_1Y_1 -2X_1Y_3\\
 &\quad -\tfrac{5}{3}X_2Y_2 +2X_2Y_3 -X_3Y_1 +3X_3Y_2 +\tfrac{4}{3}X_3Y_3 +\tfrac{3}{2}Y_1Y_2 +Y_2^2 -\tfrac{3}{2}Y_2Y_3 -Y_3^2, \quad -\tfrac{1}{3}X_1\\
 &\quad +\tfrac{1}{3}X_4 -\tfrac{1}{3}Y_2 -\tfrac{1}{3}Y_3 -\tfrac{1}{3}Y_4, 0, \dots\bigr),\\
 f_3^1 &= -b_0b_4 +b_1b_4 -\tfrac{3}{2}b_2b_3 +\tfrac{3}{2}b_3b_2 +b_4b_0 -b_4b_1 +2b_0b_2^2 -b_0b_3b_1 +2b_1b_0b_3\\
 &\quad -\tfrac{1}{2}b_1^2b_3 -2b_1b_2^2 -2b_1b_3b_0 +\tfrac{3}{2}b_1b_3b_1 -3b_2b_0b_2 +\tfrac{3}{2}b_2b_1b_2 +b_2^2b_0 +\tfrac{1}{2}b_2^2b_1\\
 &\quad +2b_3b_0b_1 -b_3b_1b_0 -b_3b_1^2 -b_0^3b_2 +b_0^2b_1b_2 +3b_0^2b_2b_0 -2b_0^2b_2b_1 -b_0b_1b_0b_2\\
 &\quad +b_0b_1^2b_2 -b_0b_1b_2b_0 -3b_0b_2b_0^2 +3b_0b_2b_0b_1 +b_0b_2b_1b_0 -\tfrac{4}{3}b_0b_2b_1^2 +2b_1b_0^2b_2\\
 &\quad -4b_1b_0b_1b_2 -3b_1b_0b_2b_0 +\tfrac{5}{3}b_1b_0b_2b_1 +\tfrac{5}{3}b_1^2b_0b_2 +\tfrac{1}{3}b_1^3b_2 +\tfrac{4}{3}b_1^2b_2b_0\\
 &\quad +2b_1b_2b_0^2 -\tfrac{11}{3}b_1b_2b_0b_1 +2b_1b_2b_1b_0 +b_2b_0^3 -2b_2b_0^2b_1 +b_2b_0b_1b_0 +\tfrac{4}{3}b_2b_0b_1^2\\
 &\quad -b_2b_1b_0^2 +2b_2b_1b_0b_1 -2b_2b_1^2b_0 -\tfrac{1}{3}b_2b_1^3 +\tfrac{3}{2}b_0^2b_1b_0b_1 -\tfrac{3}{2}b_0^2b_1^2b_0\\
 &\quad -\tfrac{3}{2}b_0b_1b_0^2b_1 +\tfrac{3}{2}b_0b_1b_0b_1^2 +\tfrac{3}{2}b_0b_1^2b_0^2 -\tfrac{3}{2}b_0b_1^2b_0b_1 +b_1b_0^3b_1\\
 &\quad -\tfrac{3}{2}b_1b_0^2b_1b_0 -\tfrac{1}{2}b_1b_0^2b_1^2 +\tfrac{3}{2}b_1b_0b_1b_0^2 -2b_1b_0b_1b_0b_1 +\tfrac{3}{2}b_1b_0b_1^2b_0\\
 &\quad +\tfrac{1}{3}b_1b_0b_1^3 -b_1^2b_0^3 +\tfrac{3}{2}b_1^2b_0^2b_1 +\tfrac{1}{2}b_1^2b_0b_1b_0 -b_1^3b_0^2 -\tfrac{1}{3}b_1^4b_0.\\[1ex]
 \xi_1^2 &= \bigl(X_1^2Y_1^2, \quad -\tfrac{3}{2}X_1^2Y_1 +X_1X_2Y_1 -X_1X_2Y_2 +\tfrac{3}{2}X_2^2Y_2 +\tfrac{1}{2}X_1Y_1^2 -X_1Y_1Y_2 -X_2Y_1^2\\
 &\quad -X_2Y_1Y_2 -\tfrac{3}{2}X_2Y_2^2, \quad \tfrac{1}{2}X_1^2 +X_1X_2 -2X_1X_3 -X_2^2 +X_2X_3 +\tfrac{1}{2}X_3^2 -\tfrac{3}{4}X_1Y_1 +\tfrac{1}{2}X_1Y_2\\
 &\quad +\tfrac{3}{4}X_1Y_3 +\tfrac{5}{4}X_2Y_1 +2X_2Y_2 -\tfrac{1}{4}X_2Y_3 -\tfrac{1}{4}X_3Y_1 -2X_3Y_2 -\tfrac{5}{4}X_3Y_3 +\tfrac{1}{2}Y_2^2 +2Y_2Y_3\\
 &\quad +\tfrac{1}{2}Y_3^2, \quad \tfrac{1}{4}X_1 -\tfrac{1}{4}X_2 -\tfrac{1}{4}X_3 +\tfrac{1}{4}X_4 -\tfrac{1}{4}Y_2 +\tfrac{1}{4}Y_4, 0, \dots\bigr),\\
 f_1^2 &= b_0^2b_3 -b_0b_2^2 -2b_0b_3b_0 +\tfrac{3}{2}b_0b_3b_1 -\tfrac{3}{2}b_1b_0b_3 +\tfrac{1}{2}b_1^2b_3 +b_1b_2^2 +\tfrac{3}{2}b_1b_3b_0\\
 &\quad -b_1b_3b_1 +2b_2b_0b_2 -2b_2b_1b_2 -b_2^2b_0 +b_2^2b_1 +b_3b_0^2 -\tfrac{3}{2}b_3b_0b_1 +\tfrac{1}{2}b_3b_1^2\\
 &\quad -b_0^2b_1b_2 +\tfrac{1}{2}b_0^2b_2b_1 +b_0b_1b_0b_2 +\tfrac{1}{4}b_0b_1^2b_2 +b_0b_1b_2b_0 -\tfrac{5}{4}b_0b_1b_2b_1\\
 &\quad -2b_0b_2b_0b_1 +b_0b_2b_1b_0 +\tfrac{3}{4}b_0b_2b_1^2 -\tfrac{3}{2}b_1b_0^2b_2 +\tfrac{7}{4}b_1b_0b_1b_2 +2b_1b_0b_2b_0\\
 &\quad -\tfrac{3}{4}b_1b_0b_2b_1 -\tfrac{3}{4}b_1^2b_0b_2 +\tfrac{1}{4}b_1^3b_2 -\tfrac{5}{4}b_1^2b_2b_0 -\tfrac{1}{4}b_1^2b_2b_1 -\tfrac{3}{2}b_1b_2b_0^2\\
 &\quad +\tfrac{9}{4}b_1b_2b_0b_1 -\tfrac{1}{4}b_1b_2b_1b_0 -\tfrac{1}{4}b_1b_2b_1^2 +\tfrac{3}{2}b_2b_0^2b_1 -b_2b_0b_1b_0 -\tfrac{5}{4}b_2b_0b_1^2\\
 &\quad -\tfrac{1}{4}b_2b_1b_0b_1 +\tfrac{3}{4}b_2b_1^2b_0 +\tfrac{1}{4}b_2b_1^3 +\tfrac{1}{2}b_1b_0^2b_1^2 +b_1b_0b_1b_0b_1 -2b_1b_0b_1^2b_0\\
 &\quad +\tfrac{1}{4}b_1b_0b_1^3 -b_1^2b_0^2b_1 +b_1^2b_0b_1b_0 -\tfrac{1}{4}b_1^2b_0b_1^2 +\tfrac{1}{2}b_1^3b_0^2 -\tfrac{1}{4}b_1^3b_0b_1\\
 &\quad +\tfrac{1}{4}b_1^4b_0.
\end{align*}
\endgroup

\medskip
\noindent\textbf{Use of AI tools.}  ChatGPT and Claude were used for exploratory computations, language editing, and the development of the fixed-point approach to swap invariance in \Cref{sec:swap-problem}.  The authors verified all statements and proofs and take full responsibility for the article.


\begin{thebibliography}{99}

\bibitem[Ba]{Ba} H.~Bachmann, \textit{An $\mathfrak{sl}_2$-action on the uri Lie algebra}, in preparation.

\bibitem[BB1]{BB1} H.~Bachmann and A.~Burmester, \textit{Combinatorial multiple Eisenstein series}, Res. Math. Sci. \textbf{10} (2023), no.~3, Paper No.~35, 32~pp.

\bibitem[BB2]{BB2} H.~Bachmann and A.~Burmester, \textit{A coproduct for multiple Eisenstein series}, in preparation.

\bibitem[BIM]{BIM} H.~Bachmann and J.-W.~van Ittersum, \textit{Formal multiple Eisenstein series and their derivations}, with an appendix by N.~Matthes, Adv. Math. \textbf{487} (2026), Paper No.~110739, 52~pp.

\bibitem[BK]{BK} H.~Bachmann and U.~K\"uhn, \textit{A dimension conjecture for $q$-analogues of multiple zeta values}, in: Periods in Quantum Field Theory and Arithmetic, Springer Proc. Math. Stat. \textbf{314}, Springer, Cham, 2020, pp.~237--258.

\bibitem[BKM]{BKM} H.~Bachmann, H.~Kanno and T.~Maesaka, \textit{Relations and derivatives of multiple Eisenstein series}, \href{https://arxiv.org/abs/2602.08176}{arXiv:2602.08176}.

\bibitem[BuK]{BuK} A.~Burmester and U.~K\"uhn, \textit{On post-Lie structures for free Lie algebras}, \href{https://arxiv.org/abs/2504.19661}{arXiv:2504.19661}.

\bibitem[BKS]{BKS} A.~Burmester, U.~K\"uhn and L.~Schneps, \textit{Moulds, bimoulds and some Lie algebras}, in preparation.

\bibitem[Bu1]{Bu1} A.~Burmester, \textit{An algebraic approach to multiple $q$-zeta values}, Ph.D. thesis, Universit\"at Hamburg, 2023.

\bibitem[Bu2]{Bu2} A.~Burmester, \textit{A generalization of formal multiple zeta values related to multiple Eisenstein series and multiple $q$-zeta values}, J. Number Theory \textbf{269} (2025), 106--137.

\bibitem[DDDHKSSV]{DDDHKSSV} D.~Dorigoni, M.~Doroudiani, J.~Drewitt, M.~Hidding, A.~Kleinschmidt, O.~Schlotterer, L.~Schneps and B.~Verbeek, \textit{Canonicalizing zeta generators: genus zero and genus one}, Comm. Math. Phys. \textbf{407} (2026), Paper No.~12, 90~pp.

\bibitem[Ec1]{Ec1} J.~\'{E}calle, \textit{ARI/GARI, la dimorphie et l'arithm\'{e}tique des multiz\^{e}tas: un premier bilan}, J. Th\'{e}or. Nombres Bordeaux \textbf{15} (2003), no.~2, 411--478.

\bibitem[Ec2]{Ec2} J.~\'{E}calle, \textit{The flexion structure and dimorphy: flexion units, singulators, generators, and the enumeration of multizeta irreducibles}, in: Asymptotics in dynamics, geometry and PDEs, generalized Borel summation, Vol.~II, CRM Series \textbf{12}, Ed. Norm., Pisa, 2011, 27--211.

\bibitem[FHK]{FHK} H.~Furusho, M.~Hirose and N.~Komiyama, \textit{Associators in mould theory}, to appear in J. Lie Theory, \href{https://arxiv.org/abs/2312.15423}{arXiv:2312.15423}.

\bibitem[Kaw1]{Kaw1} H.~Kawamura, \textit{A note on flexion units}, \href{https://arxiv.org/abs/2506.22825}{arXiv:2506.22825}.

\bibitem[Kaw2]{Kaw2} H.~Kawamura, \textit{\'{E}calle's senary relation and dimorphic structures}, \href{https://arxiv.org/abs/2509.21252}{arXiv:2509.21252}.

\bibitem[Kue]{Kue} U.~K\"uhn, \textit{Lie-algebras associated to multiple $q$-zeta values}, talk slides for the workshop \textit{Expansions, Lie Algebras, and Invariants}, CRM, Montr\'eal, July 9, 2019, \href{https://www.math.utoronto.ca/~drorbn/Talks/CRM-1907/Kuehn_montreal_2019.pdf} {available online}.

\bibitem[Ko]{Ko} N.~Komiyama, \textit{On properties of $\operatorname{adari}(\operatorname{pal})$ and $\operatorname{ganit}_{v}(\operatorname{pic})$}, \href{https://arxiv.org/abs/2110.04834}{arXiv:2110.04834}.

\bibitem[Ma]{Ma} T.~Maesaka, \textit{Flexion unit 4: properties of gaxit} (in Japanese), Mathlog article, June 19, 2026, \href{https://mathlog.info/articles/Hlr9Z4pBjCdudo1rYgu4} {mathlog.info/articles/Hlr9Z4pBjCdudo1rYgu4}.

\bibitem[Rac]{Rac} G.~Racinet, \textit{Doubles m\'elanges des polylogarithmes multiples aux racines de l'unit\'e}, Publ. Math. Inst. Hautes \'Etudes Sci. \textbf{95} (2002), 185--231.

\bibitem[Sc]{Sc} L.~Schneps, \textit{ARI, GARI, Zig and Zag: An introduction to \'{E}calle's theory of multiple zeta values}, \href{https://arxiv.org/abs/1507.01534}{arXiv:1507.01534v4}.

\end{thebibliography}
\end{document}